\documentclass[11pt]{article}
\usepackage{amsmath,amssymb,amsthm,graphicx}
\usepackage{enumitem}
\usepackage{hyperref}
\usepackage{fullpage}
\usepackage{caption}
\usepackage{subcaption}
\usepackage{float}

\newtheorem{theorem}{Theorem}[section]
\newtheorem{lemma}[theorem]{Lemma}

\newtheorem{corollary}[theorem]{Corollary}

\newtheorem{prop}[theorem]{Proposition}

\newcommand{\abc}{\mathcal{ABC}}

\begin{document}

\title{The evolution of the structure of ABC-minimal trees}

\author{Seyyed Aliasghar Hosseini\\[1mm]
\tt{sahossei@sfu.ca}\\
Department of Mathematics\\
Simon Fraser University\\
Burnaby, BC, Canada
\and
Bojan Mohar\thanks{Supported in part by an NSERC Discovery Grant (Canada),
   by the Canada Research Chair program, and by the
    Research Grant P1--0297 of ARRS (Slovenia).}~\thanks{On leave from:
    IMFM \& FMF, Department of Mathematics, University of Ljubljana, Ljubljana,
    Slovenia.}\\[1mm]
\tt{mohar@sfu.ca}\\
  Department of Mathematics\\
  Simon Fraser University\\
  Burnaby, BC, Canada
\and
Mohammad Bagher Ahmadi\\[1mm]
\tt{mbahmadi@shirazu.ac.ir}\\
Department of Mathematics\\
Shiraz University\\
Shiraz, Iran
}

\date{\today}

\maketitle

\begin{abstract}
The atom-bond connectivity (ABC) index is a degree-based molecular descriptor that found diverse chemical applications. Characterizing trees with minimum ABC-index remained an elusive open problem even after serious attempts and is considered by some as one of the most intriguing open problems in mathematical chemistry. In this paper, we describe the exact structure of the extremal trees with sufficiently many vertices and we show how their structure evolves when the number of vertices grows.
An interesting fact is that their radius is at most~$5$ and that all vertices except for one have degree at most 54. In fact, all but at most $O(1)$ vertices have degree 1, 2, 4, or 53.
Let $\gamma_n = \min\{\abc(T) :  T \textrm{ is a tree of order } n\}$.  It is shown that $\gamma_n = \tfrac{1}{365} \sqrt{\tfrac{1}{53}} \Bigl(1 + 26\sqrt{55} + 156\sqrt{106} \Bigr) n + O(1) \approx 0.67737178\, n + O(1)$.
\end{abstract}

\section{Introduction}

Molecular descriptors \cite{ref1} are mathematical quantities that describe the structure or shape of molecules, helping to predict the activity and properties of molecules in complex experiments. In the last few years a number of new molecular structure descriptors has been conceived \cite{ref2,ref3,ref4,ref5}.  Molecular descriptors play a significant role in chemistry, pharmacology, etc. Among molecular structure descriptors, topological indices have a prominent place. They are useful tools for modeling physical and chemical properties of molecules, for design of pharmacologically active compounds, for recognizing environmentally hazardous materials, etc., see~\cite{ref6}.  One of the most important topological indices is the Atom Bond Connectivity index,
also known as the ABC index.  It was introduced by Estrada \cite{ref7} with relation to the energy of formation of alkanes. It was quickly recognized that this index reflects important structural properties of graphs in general. The ABC index was extensively studied in the last few years, from the point of view of chemical graph theory \cite{ref23,ref9}, and in general graphs \cite{ref10}. Additionally, the physico-chemical applicability of the ABC index and its mathematical properties was confirmed and extended in several studies \cite{ref11,ref12,ref13,ref14,ref15,ref16,ref17}.
Some novel results about ABC index can be found in \cite{ref18,ref19,ref20} and in the references cited therein.

Let $G$ be a simple graph on $n$ vertices, and let its vertex-set be $V(G)$ and edge-set $E(G)$. By $uv$ we denote the edge connecting the vertices $u$ and $v$. The degree of a vertex $v$ is denoted by $d_v$. For an edge $uv$ in $G$, we consider the quantity
$$f(d_u,d_v )=\sqrt{\frac{d_u+d_v-2}{d_u d_v }}.$$
The \emph{Atom-Bond Connectivity index} (shortly \emph{ABC index}) of $G$ is defined as
$$\abc(G)=\sum_{uv\in E(G)} f(d_u,d_v ).$$
When the mathematical properties of a graph-based structure descriptor are investigated, one of the first questions is for which graph (with a given order $n$) is this descriptor minimal or maximal.
It is known that adding an edge in a graph strictly increases its ABC index \cite{ref21} and deleting an edge in a graph strictly decreases its ABC index \cite{ref22}. According to this fact, among all connected graphs with $n$ vertices, the complete graph $K_n$ has the maximum ABC index and graphs with minimum ABC index are trees. A tree is said to be \emph{ABC-minimal} if no other tree on the same number of vertices has smaller ABC index.

Although it is easy to show that the star graph $S_n$ has maximum ABC index among all trees of the same order \cite{ref23}, despite many attempts in the last years, it is still an open problem to characterize trees with minimum ABC-index (ABC-minimal trees).
Eventually, a computer-aided study \cite{ref24} gave rise to a conjecture on the actual structure of the ABC-minimal trees. Later results \cite{ref25,ref26} revealed that the conjecture was false, and that the true structure of the ABC-minimal trees is more complex than the computer-aided results have indicated.
In \cite{ref10}, the author presents lower and upper bounds on the ABC index of general graphs and trees, and characterizes graphs for which these bounds are best possible.

In this work we finally resolve the question on giving a precise description of ABC-minimal trees. For small values of $n$, this structure is as observed in previous works (see \cite{comp1100} for $n\le1100$). However, with number of vertices growing, a new structure, called the $C_{52}$-branch, emerges. When $n$ is very large, any ABC-minimal tree just slightly deviates from being composed of one vertex of large degree to which the $C_{52}$-branches are attached. See the last section for more details.

An interesting fact is that the radius of ABC-minimal trees is at most~$5$ (usually just 4) and that all vertices except for one have degree at most 54. In fact, all but at most $O(1)$ vertices have degree 1, 2, 4, or 53.
Let $\gamma_n = \min\{\abc(T) :  T \textrm{ is a tree of order } n\}$.  It is shown that $\gamma_n = \tfrac{1}{365} \sqrt{\tfrac{1}{53}} \Bigl(1 + 26\sqrt{55} + 156\sqrt{106} \Bigr) n + O(1) \approx 0.67737178\, n + O(1)$.

The proofs are simplified by introducing a natural equivalence relation, called similarity, on the set of all ABC-minimal trees and considering only those elements in each similarity class that are maximal in certain total order on all rooted trees. They are said to be \emph{ABC-extremal}.
The main structure results that are proved along the way towards the structural description of ABC-minimal trees are the following. For each ABC-extremal tree we let $R$ be its vertex having maximum degree. It is proved that all edges $uv$ have different vertex degrees, $d_u\ne d_v$, with a possible exception of having one edge incident with $R$, whose ends both have maximum degree, and having one edge whose ends are both of degree 2. It also has been proved that for $n\geq 415$, a 2-2 edge (known as a pendent path of length three) cannot happen in an ABC-minimal tree, see \cite{no_pendent3}. Every path starting from the root $R$ has strictly decreasing degrees, with the only possible exception when it contains an exceptional edge mentioned above (see Theorem \ref{thm:2.1} and Corollary \ref{cor:degrees decrease}). The next major step is to prove that vertices at distance 2 from $R$ have degree at most 5 (Theorem \ref{main}). Vertices of degree 1 are always adjacent to vertices of degree 2, and it is shown that the root cannot be incident with vertices of degree 2 (when $n\ge40$), and that the subtree of any vertex of degree at least 3 that is adjacent to a vertex of degree 2 has very specific structure (see Corollary \ref{cor:endingbranches}).
It is shown, as a consequence, that extremal trees are mainly composed of so called $B_3$-branches, which are grouped together into $C$-branches (see Section \ref{sect:2} for definitions). Earlier works have not expected $C$-branches, and it was conjectured in \cite{ref24} and \cite{Dim13} that $B_3$-branches are the main structural component. It was found in \cite{ourpaper} that $C$-branches will play a prominent role. A new surprise coming from our work is that for large $n$ only $C_{52}$-branches occur, see Theorem \ref{thm:large_structure}.
An interesting final outcome is that the radius of ABC-extremal trees is at most~$5$ and that all vertices except for one have degree at most 54. In fact, all but at most $O(1)$ vertices have degree 1, 2, 4, or 53.

\section{Basic ingredients}
\label{sect:2}

The following known facts will be used in the paper.

\begin{theorem}[\cite{ref27}]\label{thm:1.1}
In every ABC-minimal tree of order at least\/ $3$, each vertex of degree $1$ is adjacent to a vertex of degree $2$.
\end{theorem}

\begin{prop}[\cite{ref18}]\label{prop:2}
Let $x,y\geq 2$ be real numbers and let $a\geq 0$ and $b$, $0\leq b<y-1$ be constants. Let
$$g(x,y)= f(x+a,y-b) - f(x,y).$$
Then, $g(x,y)$ is increasing in $x$ and decreasing in $y$.
\end{prop}

\begin{theorem}[\cite{ref28}]\label{thm:1.3}
If an ABC-minimal tree has distinct vertices $v_1,v_2,v_3$ such that
$d_{v_1}\geq d_{v_2}> d_{v_3}$, then $v_3$ cannot be adjacent to both $v_1$ and $v_2$.
\end{theorem}


Let $T$ be an ABC-minimal tree of order $n$ and let $\Delta$ be the maximum degree of $T$. Let us pick one of the vertices of degree $\Delta$ and call it a \emph{root} of $T$. We will denote the root by $R$ and from now on consider any tree as a rooted tree. Thus, we can speak about descendants, predecessors, the sons of a vertex (immediate successors), etc.
We also define the \emph{height} function $h:V(T)\to{\mathbb N}$ by taking $h(v)$ to be the distance of $v$ from the root $R$ in $T$.

For each vertex $v$, we denote by $T_v$ the subtree of $T$ consisting of $v$ and all of its descendants. If $T_v\cap T_{v'}=\emptyset$, then we define another tree, $T(v,v')$, that is obtained from $T$ by exchanging $T_v$ and $T_{v'}$. Under certain conditions, this \emph{exchange operation} reduces the ABC-index, meaning that such conditions cannot occur in ABC-minimal trees. The following result was proved by Lin et al.\ in~\cite{switch_trans}.

\begin{lemma}[\cite{switch_trans}]\label{lem:similar}
Let $uv$ and $u'v'$ be edges of a tree $T$. Suppose that $v$ is a son of $u$, ~$v'$ is a son of $u'$ and that $T_v\cap T_{v'}=\emptyset$.

{\rm (a)}
If $d_u>d_{u'}$ and $d_v<d_{v'}$, then $\abc(T)>\abc(T(v,v'))$. In particular, $T$ is not ABC-minimal.

{\rm (b)}
If $d_u=d_{u'}$ or $d_v=d_{v'}$, then $\abc(T) = \abc(T(v,v'))$.
\end{lemma}

\begin{proof}
Equality in (b) is obvious since both trees have edges with same degrees. To prove (a), we apply Proposition \ref{prop:2} with $y=d_{v'}, a=0$ and $b=d_{v'}-d_v$. Observe that:
\begin{align*}
  \abc(T)-\abc(T(v,v')) & = f(d_u,d_v) + f(d_{u'},d_{v'}) - f(d_u,d_{v'}) - f(d_{u'},d_v)  \\
    & = \left( f(d_u,d_v)  - f(d_u,d_{v'})\right) - \left(f(d_{u'},d_v) - f(d_{u'},d_{v'})\right) \\
    & = g(d_{u},d_{v'}) - g(d_{u'},d_{v'}),
\end{align*}
which is positive by the proposition.
\end{proof}

Part (b) of the lemma motivates the following definitions. First of all, if the assumptions of the lemma hold and $d_u=d_{u'}$ or $d_v=d_{v'}$, then we say that $T(v,v')$ is obtained from $T$ by a \emph{similarity exchange}. Further, we say that two trees $T$ and $T'$ are \emph{similar} (or \emph{ABC-similar}) if $T'$ can be obtained from $T$ by a series of similarity exchange operations. When we treat ABC-minimal trees as rooted trees whose root is a vertex of maximum degree, we also treat any tree obtained by taking a different vertex of maximum degree as the root as being similar.  Note that similarity is an equivalence relation that preserves the ABC-index.  In order to characterize ABC-minimal trees, it suffices to describe one tree in each similarity class. Below we will introduce some special properties of ABC-minimal trees that will define a subclass called ABC-extremal trees.

For our next definition we will need a special linear ordering $\succ$ among the isomorphism classes of all rooted trees (with at most $n$ vertices). For two such trees $T$ and $T'$, we first compare their roots. If the root of $T$ has larger degree than the root of $T'$, then we set $T\succ T'$ (and we set $T'\succ T$ if the root of $T'$ has larger degree).  If the degrees are the same, both equal to $d\ge0$, we lexicographically compare their subtrees $T_1,\dots,T_d$ and $T_1',\dots,T_d'$ rooted by the sons of their roots, and we set $T\succ T'$ if the subtrees of $T$ are lexicographically larger. Note that these subtrees are lexicographically the same if and only if $T$ and $T'$ are isomorphic. (This can be easily proved by induction.)
We write $T\succeq T'$ if either $T\succ T'$ or $T$ and $T'$ are isomorphic as rooted trees.


An ABC-minimal tree $T$ is said to be \emph{ABC-extremal} if it is $\succ$-largest in its similarity class which is the same as greedy tree for a given degree sequence. Note that  this also included the best choice of the root among the vertices of maximum degree.\footnote{It will be shown later that there are at most two such vertices in any ABC-minimal tree.} The ABC-extremal trees have some additional properties that will be useful for us. Let us summarize some of them.

Let $T$ be an ABC-extremal tree. Then $T$ has the following properties:
\begin{itemize}
  \item[\rm (P1)] Let $u,v\in V(T)$. Suppose that $h(u)<h(v)$ and if $u$ is not the root then $v\notin T_u$. Then $T_{u}\succeq T_{v}$ and, in particular, $d_{u}\ge d_{v}$.
  \item[\rm (P2)] Let $h_2 = \min \{h(v)\mid d_v=2 \}$. If $T$ contains an edge $uu'$ such that $d_u = d_{u'} = 2$, then either $h(u)=h_2$ or $h(u')=h_2$.
  \item[\rm (P3)] Suppose that $u,v$ are non-root vertices with the same degree, $d_u=d_v$, and let $r=d_u-1$. Let $u_1,\dots,u_r$ be the sons of $u$ and let $v_1,\dots,v_r$ be the sons of $v$. Suppose that $T_{u_1}\succeq T_{u_2}\succeq \cdots \succeq T_{u_r}$ and $T_{v_1}\succeq T_{v_2}\succeq \cdots \succeq T_{v_r}$. If $h(u)<h(v)$, then $T_{u_r}\succeq T_{v_1}$. If $h(u)=h(v)$, then either $T_{u_r}\succeq T_{v_1}$ or $T_{v_r}\succeq T_{u_1}$. Assuming that $T_{u_r}\succeq T_{v_1}$, then we have, in particular, that $d_{u_1}\ge d_{u_2}\geq \cdots \geq d_{u_r} \geq d_{v_1}\geq d_{v_2}\geq \cdots \geq d_{v_r}$.
\end{itemize}

\begin{proof}
(P1) The proof is by induction on $h(u)$. If $u$ is the root, then the property is clear by the definition of similarity which includes exchanging the root with another vertex of maximum degree if that rooted tree is $\succ$-larger. Therefore we may assume that $u$ is not the root. Let $\hat u$ and $\hat v$ be the predecessors of $u$ and $v$, respectively. By the induction hypothesis, we have that $d_{\hat u} \ge d_{\hat v}$. Suppose, for a contradiction, that $T_u \prec T_v$. In particular, $d_u\le d_v$. Since $v\notin T_u$, we have $T_{u}\cap T_{v} = \emptyset$. By Lemma \ref{lem:similar}(a), we conclude that either $d_{\hat u} = d_{\hat v}$ or $d_u=d_v$. Therefore, $T(u,v)$ is obtained from $T$ by a similarity exchange. Since $T_v \succ T_u$, we conclude that $T(u,v) \succ T$, which contradicts the assumption that $T$ is ABC-extremal.

(P2) As proved in \cite{ref29} (see Lemma \ref{lem:2-2edge}), there is at most one such 2-2 edge in $T$. Suppose that $h(u)<h(u')$, and let $v$ be a degree-2 vertex with $h(v)=h_2$. If $h(u)>h_2$, then (P1) implies that $T_v\succeq T_u$. By Theorem \ref{thm:1.3} we see that the son of $v$ has degree 1 and since the son $u'$ of $u$
has degree 2, we have that $T_u\succ T_v$, a contradiction.

(P3) If $T_{u_r}\prec T_{v_1}$ and either $h(u)<h(v)$ or $T_{v_r}\prec T_{u_1}$, then one of the similarity exchanges $T(u_r,v_1)$ or $T(v_r,u_1)$ would yield a $\succ$-larger tree. This contradiction shows that (P3) holds.
\end{proof}

ABC-extremal trees and their properties (P1)--(P3) have been used frequently in previous works and were sometimes called ``greedy trees''.

At the end of the next section, we will show that (P1) holds also when $v$ is a successor of $u$.

\begin{figure}[htb]
    \centering
    \includegraphics[width=0.65\textwidth]{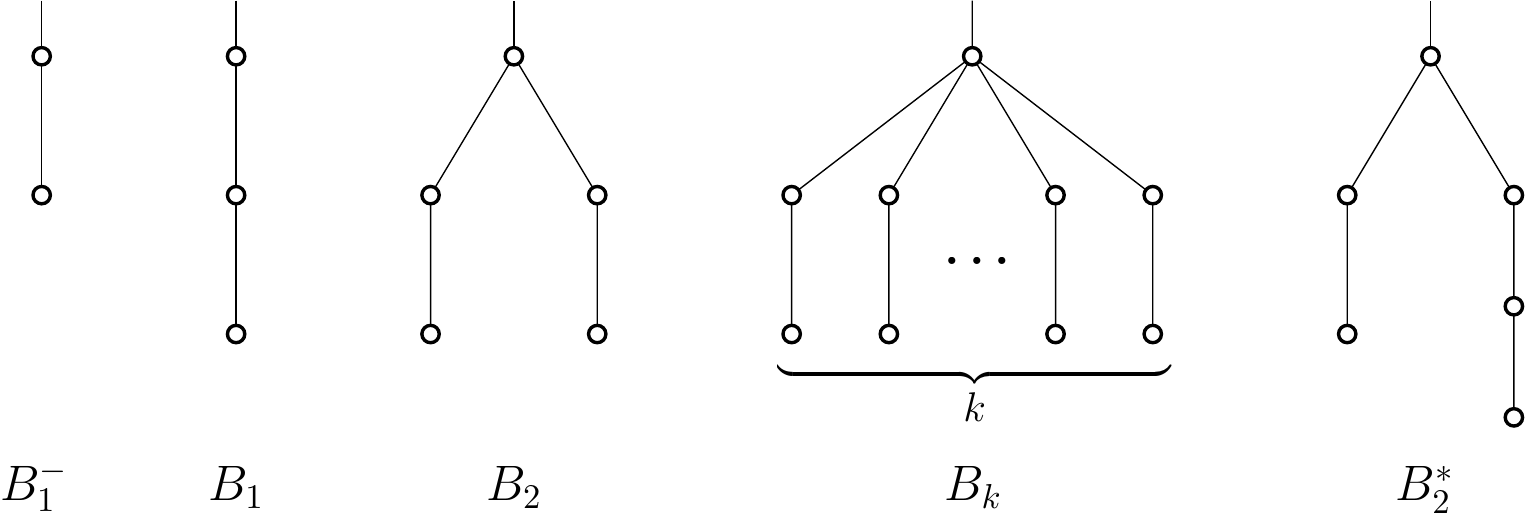}
    \caption{Definition of $B_k$-branches.}
    \label{fig:B_k}
\end{figure}

Suppose that all sons of a vertex $v\in V(T)$ are of degree 2 and all second descendants are of degree 1. Then the subtree $T_v$ is said to be a \emph{$B_k$-branch}\footnote{In some earlier papers, $B_1^-$-branches are called  $B_1$-branches.} with root $v$, where $k$ is the number of sons of $v$ (which is $d_v-1$ if $v\ne R$). See Figure \ref{fig:B_k}.  In our later figures, we will represent each $B_k$-branch by a triangle with the number $k$ next to it. If $k=3$, the number may be omitted.

Simulations and existing results \cite{ref18,ref19,ref20} show that ABC-minimal trees have lots of $B_3$-branches and only a small number of $B_k$-branches for $k\ne3$; see also our Corollary \ref{cor:endingbranches}.

If $T_v$ can be obtained from a $B_k$-branch by adding a new vertex and joining it to one of vertices of degree 1, then we say that this is a $B_k^*$-branch. This kind of subtrees may appear in ABC-minimal trees, but cannot occur more than once (see \cite{ref29}). The reason is the following result about \emph{2-2 edges} (edges whose both ends have degree 2, also known as pendent path of length three).

\begin{lemma}[\cite{ref29, no_pendent3}]
\label{lem:2-2edge}
Any ABC-minimal tree has at most one 2-2 edge and, if there is one, it is part of a $B_1$-branch. Furthermore, ABC-minimal trees of order $\geq415$ contain no 2-2 edges.
\end{lemma}

For a tree $T$, let $B(T)$ be the set of those vertices different from the root that are of degree at least 3 that have a son of degree 2, whose son is of degree 1. Note that for every $k\ge1$, $B(T)$ contains all roots of $B_k$-branches and $B_k^*$-branches ($k\ge 2$). If a vertex $u\in B(T)$ is not a root of some $B_k$-branch, then it is referred to as an \emph{exceptional vertex} in $B(T)$ and its subtree $T_u$ is called \emph{$B$-exceptional branch}. Note that the root of any $B_k^*$-branch with $k\ge2$ is an exceptional vertex. Our next result shows that any ABC-extremal tree has at most one exceptional vertex in $B(T)$.

\begin{lemma}\label{lem:B}
Any ABC-extremal tree $T$ has at most one exceptional vertex in $B(T)$. Moreover, if $v$ is exceptional, then $d_v = \max\{d_u \mid u\in B(T)\}$.
\end{lemma}

\begin{proof}
Theorem \ref{thm:1.1} implies that a vertex in $B(T)$ cannot have a son of degree 1. Thus, an exceptional vertex either has a son of degree at least 3, or it has a son of degree 2 that is not incident to a vertex of degree 1.

Suppose that there are two exceptional vertices, $v_1$ and $v_2$, where $v_i$ has a son $v_i'$ of degree 2 (whose son is a degree-1 vertex) and also has a son $u_i$ that is not the father of a vertex of degree 1 ($i=1,2$). We select $u_i$ to be of degree more than 2 if possible. Suppose that $d_{v_1}\geq d_{v_2}$.
Lemma \ref{lem:similar} shows that $\abc(T(v_1',u_2))\le \abc(T)$, where the inequality is strict unless $d_{v_1}=d_{v_2}$ or $d_{u_2}=d_{v_1'}=2$. Since $T$ is ABC-minimal, we have one of the two equalities. In either case, replacing $T$ with the tree $T(v_1',u_2)$ is a similarity exchange. If $d_{v_1}>d_{v_2}$, then this exchange gives a $\succ$-larger tree, contradicting extremality of $T$. The same may give a contradiction if $d_{v_1}=d_{v_2}$; but if it does not, then we consider $T(v_2',u_1)$, and it is easy to see that this yields a $\succ$-larger tree. This shows that there is at most one exceptional vertex.

Suppose now that $T$ has precisely one exceptional vertex $v\in B(T)$. If $d_v$ is not the largest in $\{d_u\mid u\in B(T)\}$, doing a similar exchange with the vertex $u$ in $B(T)$ of maximum degree gives us a contradiction to the extremality of $T$.
\end{proof}

We will show in Lemma \ref{no_b_exception} that in addition to $B_k^*$ ($k\ge2$) only one type of $B$-exceptional branches may exist in any ABC-minimal tree.

Our next goal is to show that $B_k$-branches may occur only for $k\le5$.

\begin{lemma}[\cite{ref18,DuFonseca}]
\label{B_k branches are small}
If an ABC-minimal tree contains a $B_k$-branch, then $k\le 5$. If it contains a $B_k^*$-branch, then $k\le3$.
\end{lemma}

The first claim in the lemma was essentially proved in \cite{ref18} with a different approach, but the proof uses some additional assumptions that we do not have. The second claim about $B_k^*$-branches can be found in \cite{DuFonseca}. We include a sketch of our own proof, some of whose easier details are left to the reader.

\begin{proof}
Let $u$ be the root of a $B_k$-branch ($B_k^*$-branch) considered. We may assume that $u$ is not the root. Let $\hat u$ be the father of $u$. For $B_k$ ($k\ge 6$), we replace $T_u$ with $B_{k-4}$ and $B_3^*$, both attached to $\hat u$. Note that the degree of $\hat u$ increases by one. Let $T'$ be the resulting tree. Now it is easy to see that $\abc(T) - \abc(T')>0$, which is a contradiction.

Similarly, for $B_k^*$ ($k\ge5$), we replace $T_u$ with $B_{k-3}$ and $B_3$ attached to $\hat u$. And for $B_4^*$, we replace $T_u$ with the tree $B_3^{**}$ shown in Figure \ref{B32}. Details are omitted.
\end{proof}

Dimitrov \cite{ref18} also proved that $B_5$-branches can be excluded under the assumption that there is a $B_2$ or $B_3$-branch as a sibling\footnote{For example, the case where the root has only $C_k$ branches and one $B_5$ as its children is not considered in \cite{ref18}.}. Below we give a slightly stronger result.

\begin{lemma}\label{B_k_equal}
Let $T$ be an ABC-minimal tree.
If a $B_k$-branch and a $B_l$-branch are siblings, then $|k-l|\le 1$.
\end{lemma}

\begin{proof}
Suppose that $k\ge l$ and let $t=k-l$.
Let us assume that a $B_k$-branch and a $B_{k-t}$-branch exist as siblings in $T$. Let the parent of $B_k$ and $B_{k-t}$ be a vertex of degree $d$. Theorem \ref{thm:1.3} implies that for every path starting at the root, the vertex-degrees along the path never increase, thus we have $d\ge k+1$.
By detaching one vertex of degree 2 from $B_k$ and attaching it to $B_{k-t}$ we obtain a tree $T'$ in which $B_k$ is replaced by $B_{k-1}$ and $B_l$ with $B_{l+1}$. Since $f(2,x)=\sqrt{2}/2$ is independent of $x$, we have
$$\abc(T)-\abc(T') = f(d,k+1)+f(d,k-t+1) - f(d,k) - f(d,k-t+2).$$
For fixed $d$ and $k$, this difference is decreasing in terms of $k-t+1$ (by Proposition \ref{prop:2} used on the second and the last term with $a=0$, $b=1$ and $y=k-t+2$). This means that the difference is (strictly) increasing in terms of $t$. Since $\abc(T)-\abc(T') = 0$ when $t=1$, we conclude that for $t\geq 2$  the difference is positive and we can apply the suggested change to obtain a contradiction to ABC-minimality of $T$.
\end{proof}

As in the above proof, we will frequently compare the ABC-index of a tree $T$ with that of a modified tree $T'$. To make the notation shorter we will write
\[\Delta(T,T') = \abc(T) - \abc(T').\]

We define a \emph{$C_k$-branch} as a subtree $T_v$, in which $v$ has precisely $k$ sons $v_1,\dots,v_k$, and their subtrees $T_{v_1},\dots,T_{v_k}$ are all $B_3$-branches. In our figures, we will represent a $C_k$-branch as a square with $k$ written inside the square.

\begin{lemma}[\cite{no_pendent3}] \label{lemma:C}
Let $T$ be an ABC-minimal tree.
If there are $C_k$-branch and $C_l$-branch as siblings, then $|k-l|\le 1$.
\end{lemma}

\begin{lemma}\label{lem:C_k}
No ABC-minimal tree contains a $C_k$-branch with $k\geq 143$.
\end{lemma}

\begin{proof}
Assume that $T$ is an ABC-minimal tree with a $C_k$-branch, where $k\geq 143$. We will assume that $k$ is odd. For the even case only some small modifications are needed. We can replace the $C_k$-branch with two $C_{k'}$-branches, where $k'=\frac{k-1}{2}$, see Figure \ref{fig:C_k}. More precisely, the $B_3$-branches within $C_k$ are divided evenly between the two $C_{k'}$-branches, and the remaining $B_3$ is replaced by three paths attached to three $B_3$ branches (which turns them into $B_4$) as indicated in Figure \ref{fig:C_k}.  Let $T'$ be the resulting tree. We have:

\begin{figure}[htb]
    \centering
    \includegraphics[width=0.85\textwidth]{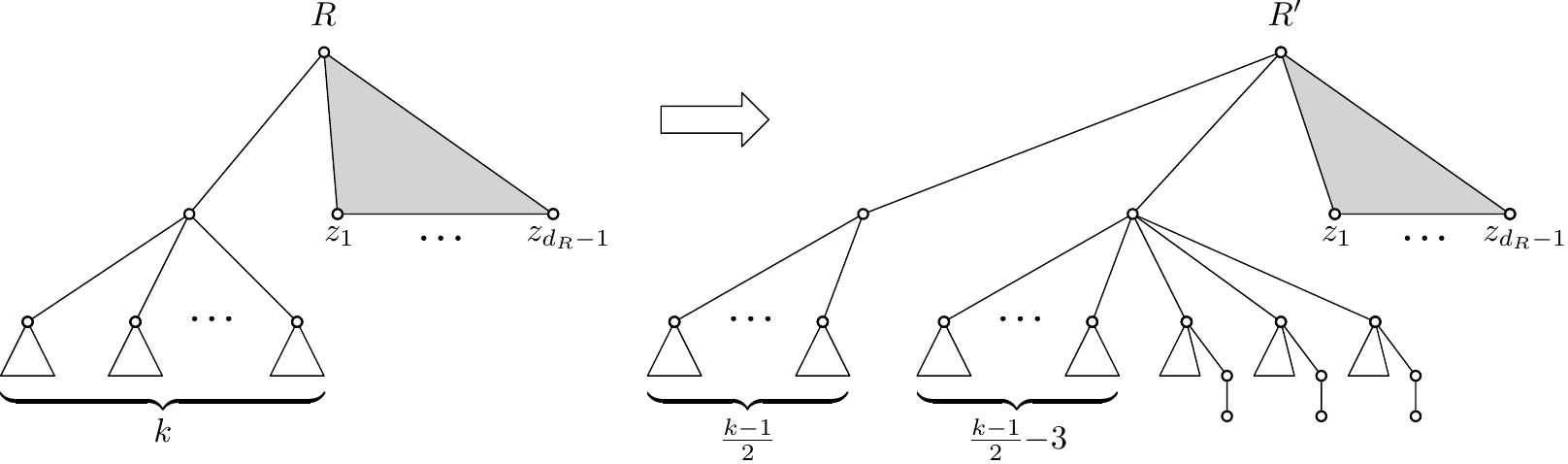}
    \caption{Suggested change when there exists a $C_k$ branch with odd $k\geq 143$.}
    \label{fig:C_k}
\end{figure}

\begin{eqnarray*}
\Delta(T,T') &=& f(d_R,k+1 )+kf(k+1,4 )+\sum_{i=1}^{d_R-1} f(d_R,d_{z_i}) - 2f(d_R+1,\tfrac{k+1}{2}) \\
 && -(k-4)f(\tfrac{k+1}{2},4)-3 f(\tfrac{k+1}{2},5) -\sum_{i=1}^{d_R-1} f(d_R+1,d_{z_i}).
\end{eqnarray*}

Using Proposition \ref{prop:2} we can see that  $f(d_R,d_{z_i}) - f(d_R+1,d_{z_i})$ is increasing in $d_{z_i}$. Thus, to have the worst case we may consider the lowest possible values for the degrees $d_{z_i}$.
Note that there are $B_3$-branches in $C_k$ and since $k<d_R$, Lemma \ref{lem:similar} shows that $d_{z_i}\geq 4$. So:
\begin{eqnarray*}
 \Delta(T,T') &\geq& f(d_R,k+1 )+kf(k+1,4 )+(d_R-1) f(d_R,4) - 2f(d_R+1,\tfrac{k+1}{2}) \\
 && -(k-4)f(\tfrac{k+1}{2},4) -3 f(\tfrac{k+1}{2},5)-(d_R-1) f(d_R+1,4).
\end{eqnarray*}

Now, let us rewrite this inequality as follows:
\begin{eqnarray*}
 \Delta(T,T') &\geq& f(d_R,k+1)-f(d_R+1,\tfrac{k+1}{2}) + \\
  && (d_R-1)(f(d_R,4)-f(d_R+1,4)) + \\
 && k(f(k+1,4)-f(\tfrac{k+1}{2},4)) + \\
 && 3(f(\tfrac{k+1}{2},4)-f(\tfrac{k+1}{2},5)) + \\
 && f(\tfrac{k+1}{2},4)-f(\tfrac{k+1}{2},d_R+1).
\end{eqnarray*}

Again, using Proposition \ref{prop:2} we can see that the value in each line except the first one is increasing in $k$. Regarding the first line, observe the following:
\begin{eqnarray*}
f(d_R,k+1)-f(d_R+1,\tfrac{k+1}{2}) & = & f(d_R,k+1)-f(d_R+1,k+1)\\
 && + \ f(d_R+1,k+1)-f(d_R+1,k)\\
 && + \ f(d_R+1,k)-f(d_R+1,k-1)\\
 && + \ \ldots\\
 && + \ f(d_R+1,\tfrac{k+3}{2})-f(d_R+1,\tfrac{k+1}{2}).
\end{eqnarray*}
Here, each line is increasing in $k$, therefore $f(d_R,k+1)-f(d_R+1,\frac{k+1}{2})$ is also increasing in $k$. If we substitute $k$ by 143, then the above lower bound only depends on one variable, $d_R$, and it is easy to check that $\Delta(T,T')>0$ for any value of $d_R\geq k$. Therefore, $\Delta(T,T')>0$ for any $d_R\geq k\geq 143$.
\end{proof}

Note that the above proof has some room for improvement, since we have considered exclusive extreme configurations, in one assuming that $d_{z_i}=4$ and also considered only one copy of a $C_k$ branch.

\begin{figure}[htb]
    \centering
    \includegraphics[width=0.7\textwidth]{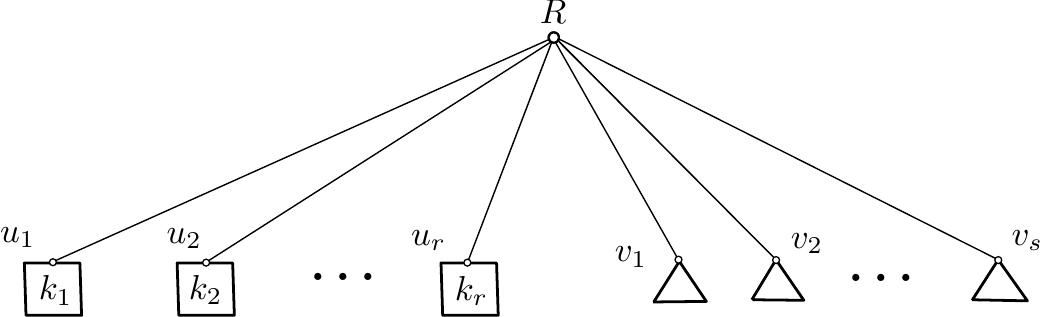}
    \caption{The basic structure (after deletion of a small number of vertices). When $n$ is small, we have only $B_3$-branches ($r=0$). As $n$ grows, a combination of both occurs and when $n$ is sufficiently large, only $C_{k_i}$-branches remain ($s=0$), eventually with all $k_i$ being equal to 52 and $r=n/365 - O(1)$.}
    \label{fig:main_struct}
\end{figure}

In this paper it will be proved that ABC-minimal trees have the structure close to that shown in Figure \ref{fig:main_struct} in the sense that there is a small number of vertices whose deletion gives us this form. Moreover, the following transition occurs. Let us denote by $r$ the number of $C_k$-branches (whose roots $u_1,\dots,u_r$ are adjacent to the root $R$) and by $s$ the number of $B_3$-branches, whose roots $v_1,\dots,v_s$ are adjacent to $R$.  When $n$ is relatively small, we have no $C_{k_i}$-branches ($r=0$). In the intermediate range between around a 1000 and several thousands, we have a combination of both extremes, depending on the remainder of $n$ divided by 365. When $n$ is sufficiently large, it turns out that $B_3$-branches disappear ($s=0$) and all values $k_i$ stabilize at $52$, with a few exceptions (for which $k_i=51$ or 53; see Lemma \ref{lemma:C}).

\section{Degrees strictly decrease away from the root}
\label{sect:3}

Let $T$ be an ABC-minimal tree of order $n$ and let $\Delta$ be the maximum degree of $T$. Theorem \ref{thm:1.3} implies that for every path starting at the root, the vertex-degrees along the path never increase. The goal of this section is to prove that the degrees are strictly decreasing, with two sporadic exceptions.

\begin{theorem}\label{thm:2.1}
Let\/ $T$ be an ABC-minimal tree of order greater than 9 and maximum degree $\Delta$. For every $k\ge2$, $T$ contains at most one edge, whose end vertices both have degree $k$. Moreover, if such an edge exists, then $k$ is either $2$ or $\Delta$.
\end{theorem}

\begin{proof}
ABC-minimal trees of order $\leq 1100$ are known (see \cite{comp1100}) and they satisfy Theorem \ref{thm:2.1}.  It is also known that for trees of order $\geq 415$ there are no 2-2 edges in ABC-minimal trees (see Lemma \ref{lem:2-2edge}). Since Theorem \ref{thm:2.1} holds for trees of smaller order we may assume that $k>2$ and that there exist an edge $uv$ whose end vertices have the same degree, $d_u=d_v=k$. Suppose that $u$ is closer to the root than $v$. If there is more than one such edge, consider the one with the highest value of $k$ and if there is more than one such $k$-$k$ edge, consider one which is farthest from the root. Then all descendants of $v$ have degree smaller than $k$. Detach the child $x$ of $v$ with the largest degree (together with its subtree $T_x$) and connect it to $u$ as shown in Figure \ref{fig:no_kk}.
(The subtree $T_x$ is not shown in the figure.)
Let $T'$ be the resulting tree. Note that by selection of the edge $uv$, we know that $d_x < k$. 
We have two cases. If $u$ is the root, then if $uv$ is the only edge whose end vertices have degree $k$, then Theorem \ref{thm:2.1} holds for $k=\Delta$. If there is another vertex $w$ with $d_w=k$, then it should be adjacent to $u$ and we can consider $w$ to play the role of $R$ in Figure \ref{fig:no_kk}, so we have $d_R\geq d_u$. If $u$ is not the root, then let $R$ be the parent of $u$ and therefore $d_R\geq d_u$.

\begin{figure}[htb]
    \centering
    \includegraphics[width=0.7\textwidth]{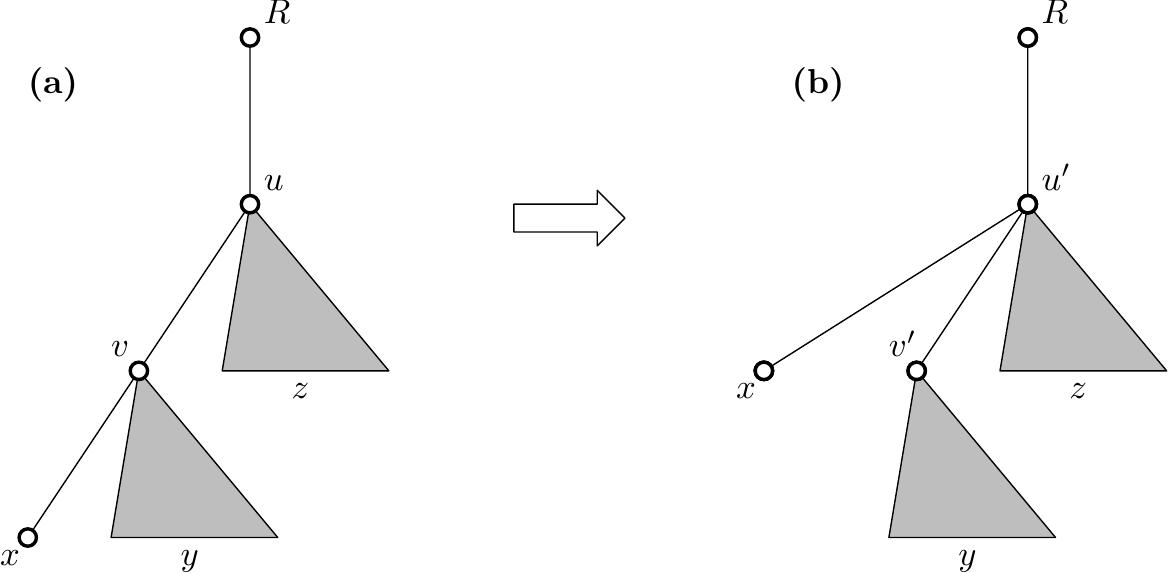}
    \caption{Changing the tree when $2<d_u=d_v\le\Delta$.}
    \label{fig:no_kk}
\end{figure}
Let $y_1,\dots,y_{k-2}$ be the children of $v$ different from $x$ and let $z_1,\dots,z_{k-2}$ be the children of $u$ different from $v$.
We have selected $x$ so that $d_x\geq d_{y_j}$. Note that if there exist $i$ such that $d_x\geq d_{z_i}$ then we can exchange the branch rooted at $z_i$ with the branch rooted at $x$ without changing the ABC index of the tree. So without loss of generality we can assume that $d_{z_i}\geq d_x\geq d_{y_j}$ ($1\le i,j\le k-2$).
Figure \ref{fig:no_kk} only shows edges whose degrees have changed. Let $\alpha(T)$ and $\alpha(T')$ be the contribution of all these adges to $\abc(T)$ and $\abc(T')$, respectively. Clearly,
\begin{eqnarray*}
 \alpha(T) &=& f(d_R,d_u )+f(d_u,d_v )+\sum_{i=1}^{k-2} f(d_u,d_{z_i})+\sum_{j=1}^{k-2} f(d_v,d_{y_j})+f(d_v,d_x) \\
 &=& f(d_R,k)+f(k,k)+\sum_{i=1}^{k-2} f(k,d_{z_i}) +\sum_{j=1}^{k-2} f(k,d_{y_j}) +f(k,d_x).
\end{eqnarray*}
Similarly,
\begin{eqnarray*}
\alpha(T') &=&
f(d_R,k+1)+f(k+1,k-1)+\\
&&\sum_{i=1}^{k-2} f(k+1,d_{z_i})+\sum_{j=1}^{k-2} f(k-1,d_{y_j})+f(k+1,d_x).
\end{eqnarray*}

Note that there exists $z\in \{z_1,\dots,z_{k-2}\}$ such that $f(k,d_{z_i}) - f(k+1,d_{z_i}) \geq f(k,d_z) - f(k+1,d_z)$ for all $i=1, \ldots, k-2$. Considering a similar inequality for $y\in \{y_1,\dots,y_{k-2}\}$, we have:
\begin{eqnarray*}
  \alpha(T) - \alpha(T') &\geq& f(d_R,k)-f(d_R,k+1)+f(k,k)-f(k+1,k-1)+\\
  && f(k,d_x)-f(k+1,d_x)+ (k-2)f(k,d_z)-(k-2)f(k+1,d_z)+\\
  && (k-2)f(k,d_y)-(k-2)f(k-1,d_y).
\end{eqnarray*}

We have discussed that $d_R\geq k\geq d_z\geq d_x\geq d_y$ and we would like to show that $\alpha(T) - \alpha(T') > 0$, i.e., this change improves the ABC-index. By Proposition \ref{prop:2},
$f(d_R,k)-f(d_R,k+1)$ is increasing in $d_R$. Since $d_R\geq k$, this implies that
$$f(d_R,k)-f(d_R,k+1)\geq f(k,k)-f(k,k+1).$$

Similarly, we have:
$f(k,d_x)-f(k+1,d_x)$ is increasing in $d_x$,
$f(k,d_z)-f(k+1,d_z)$ is increasing in $d_z$,
and $f(k,d_y)-f(k-1,d_y)$ is decreasing in $d_y$.
Therefore we may replace $d_z$ and $d_y$ by $d_x =:m < k$, so we have:
\begin{eqnarray}\label{eqn:1}
  \alpha(T) - \alpha(T') &\geq& f(k,k)-f(k,k+1)+f(k,k)-f(k+1,k-1)+  \nonumber \\
&& f(k,m)-f(k+1,m) + (k-2)f(k,m)-(k-2)f(k+1,m)+  \nonumber \\
&& (k-2)f(k,m)-(k-2)f(k-1,m). \label{eq:s1}
\end{eqnarray}

If $m=1$, then it follows by Theorem \ref{thm:1.1} that $k=2$ and we have settled this case before. So we may assume that $m\geq2$.
Therefore $1<m<k$.
Using computer, we have calculated the values of the right-hand side of (\ref{eq:s1}) for all pairs $(m,k)$, where $1<m<k\le 10^5$, and
$k\geq 5$. The same was checked for $k=4$ when $m=2$. In all cases the computation confirms that $\Delta(T,T')>0$.

\begin{figure}[!htb]
    \centering
    \includegraphics[width=0.5\textwidth]{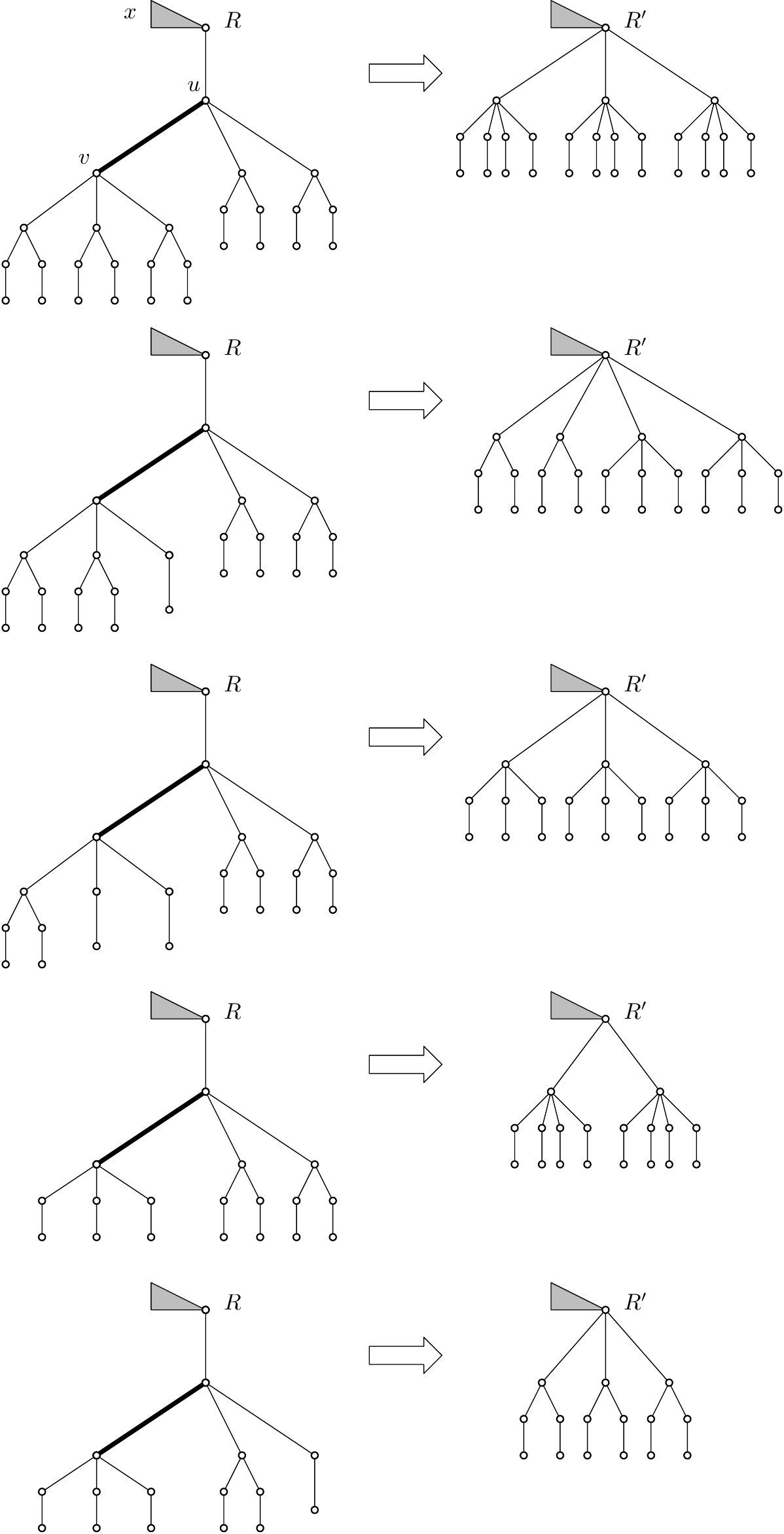}
    \caption{Changing the tree when $k=4$ and $m\neq 2$.}
    \label{fig:k4}
\end{figure}

Suppose now that $k>10^5$.
Let $m=ck$ where $\frac{1}{k}<c<1$. Using Taylor series we can expand the right-hand side of (\ref{eq:s1}) in terms of $k$ (factor out $1/\sqrt{8k^3}$ and then use Taylor series of order 5 to expand).\footnote{The expansion was produced with the help of the software platform Maple, version 18.} We also made the substitution $A=\sqrt{\frac{c}{c+1}}$ and have obtained the following:
\begin{eqnarray}
  \alpha(T) - \alpha(T') &\geq&
      \frac{1}{\sqrt{8k^3}}
      \left( 1 - \sqrt{2}\, A + \frac{\sqrt{2}}{2}\, A^3\right) +
  \nonumber \\
  &&
  \frac{3}{2\sqrt{k^5}} \left(\frac{2}{3} \frac{A}{c} - \frac{35}{24} + A - \frac{1}{4} A^3 - \frac{A}{c+1} + \frac{3}{2}\,\frac{A^3}{c+1}\right) + O(k^{-7/2}).
  \label{eq:s2}
\end{eqnarray}
Note that $A$ is bounded, $0<A<1$.

First, we need to show that the coefficient of $k^{-\frac{3}{2}}$ is positive. When $\frac{1}{k}<c<1$ we have $\frac{1}{k+1}<\frac{c}{c+1}<\frac{1}{2}$. This implies that $1-\sqrt{2}A+\frac{\sqrt{2}}{2}A^3 > 1 - \sqrt{2}/2 > 0$. By rewriting (\ref{eq:s2}), we obtain:
\begin{equation}
  k^{3/2}(\alpha(T) - \alpha(T')) >
      \frac{\sqrt{8} - 2}{8} + \frac{A}{m} + O(\tfrac{1}{k}).
  \label{eq:s3}
\end{equation}
The order of the middle term in (\ref{eq:s3}) is $O(k^{-1/2})$ and is positive.
Note that $c+1$ and $\frac{c}{c+1}$ are bounded and the only case that can cause problem is when $c$ is in the denominator and $c\sim \frac{1}{k}$, which means $m=O(1)$. For $k\ge 10^5$, the terms are negligible in comparison with the constant value of the first term. This shows that $\alpha(T) - \alpha(T') > 0$ for $k>10^5$ and any value of $m<k$.

There are some remaining cases for small values of $k$ ($k=4$ and $k=3$).  When $k=4$ and $m=2$ it is easy to check that (\ref{eqn:1}) is positive. So we may assume that $k=4$ and $m\neq2$, or $k=3$. For these cases note that since equation (\ref{eqn:1}) is not positive, we cannot only consider the worst case and need to discuss all possible values of $d_{z_i}, d_x$ and $d_{y_i}$. Figure \ref{fig:k4} deals with the case when $k=4$. One can check that using the suggested change of the tree, the ABC-index becomes smaller.  Note that since $u$ has a neighbor of degree 3 ($m\neq 2$), all neighbors of $R$ have degree $\geq 3$, otherwise we can exchange them and get a tree with smaller ABC-index ($d_R\geq d_u$).

As a case in point we will discuss the first case shown in Figure \ref{fig:k4} and leave the rest to the reader. As before, we let $\alpha(T)$ be the contribution to $\abc(T)$ of all edges that are shown in the figure. We have:
\begin{eqnarray*}
 \alpha(T) &=& f(d_R,4)+5f(4,3)+f(4,4)+10f(3,2)+10f(2,1)\quad \hbox{and} \\
 \alpha(T') &=& 3f(d_R+2,5)+12f(5,2)+12f(2,1).
\end{eqnarray*}
The change in ABC-index when passing from $T$ to $T'$ is equal to $\alpha(T)-\alpha(T')$ plus all differences $f(d_R,d_x)-f(d_R+2,d_x)$ for each neighbor $x$ of $R$ different from $u$. By Proposition \ref{prop:2}, and since $d_x \ge 3$, we have
$$f(d_R,d_x)-f(d_R+2,d_x)\geq f(d_R,3)-f(d_R+2,3).$$
Consequently,
\begin{eqnarray}
\Delta(T,T')\geq \alpha(T)-\alpha(T')+(d_R-1)\left(f(d_R,3)-f(d_R+2,3)\right). \label{example}
\end{eqnarray}

Using (\ref{example}), it is easy to check that $\Delta(T,T')>0$ for $d_R\geq 4$.


The second case which needs specific treatment is when $k=3$. In this case $d_{y}=d_x=2$ and $d_{z}=2$ or $3$. To solve this case we will first show that a vertex of degree 3 cannot have two descendants of degree 3. For a contradiction assume $u$ has two degree-3 descendants.
Since the edge $uv$ is taken farthest from the root, the descendants of these degree-3 vertices have degree 2, and all further descendants have degree 2 or 1. Having additional descendants of degree 2 does not affect the computation in the sequel, thus we may assume that the situation is as shown in Figure \ref{fig:R33}.
Note that if $u$ has only one descendent of degree three (and one descendent of degree two) and there is another 3-3 edge in the graph, then we have two cases. If $d_R=3$, then another descendent of $R$ is either of degree two or three and in both of these cases by an exchange we will get the structure shown in Figure \ref{fig:R33}. If there is another disjoint 3-3 edge in the graph, then one of those degree three vertices will be the root of a $B_2$ that can be exchanged with $z$ (without changing the ABC-index), and we can then make the change shown in Figure \ref{fig:R33}. Also note that $d_R\geq3$; otherwise we would get a small tree and small trees are known to satisfy our theorem \cite{comp1100}. 
Now it is easy to check that the suggested structure will have smaller ABC-index, thus yielding a contradiction.

\begin{figure}[htb!]
    \centering
    \includegraphics[width=0.45\textwidth]{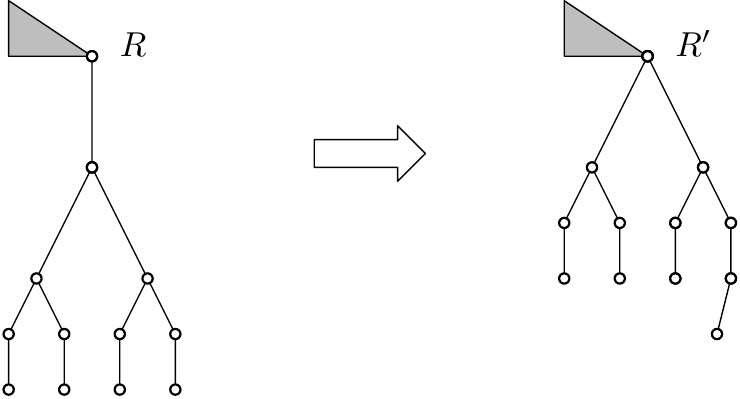}
    \caption{Changing the tree when two sons of $u$ have degree 3.}
    \label{fig:R33}
\end{figure}

\begin{figure}[htb!]
    \centering
    \includegraphics[width=0.42\textwidth]{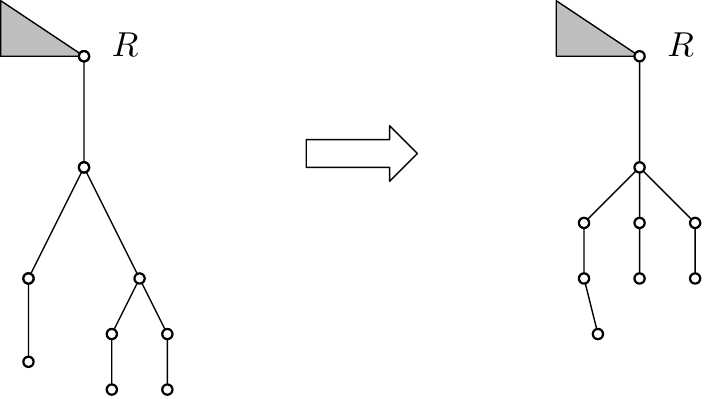}
    \caption{Changing the tree when $d_R\geq 5$.}
    \label{fig:R3}
\end{figure}

\begin{figure}[htb!]
    \centering
    \includegraphics[width=0.5\textwidth]{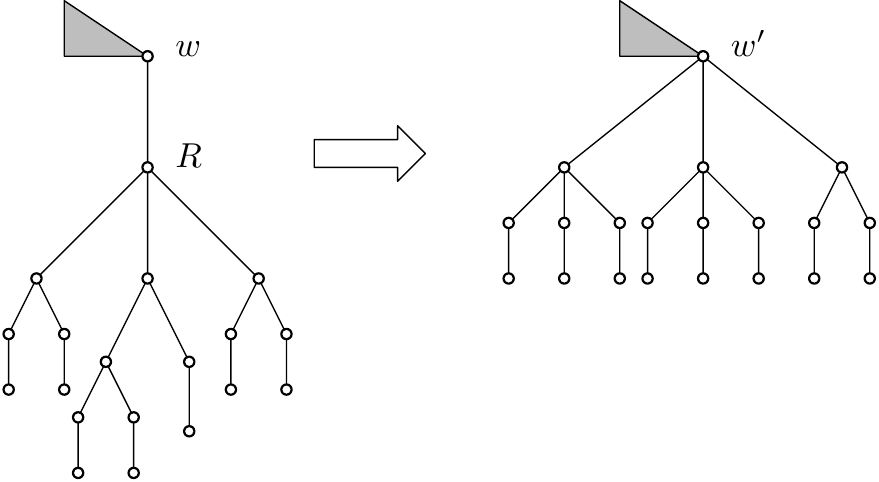}
    \caption{Changing the tree when $d_R=4$.}
    \label{fig:R4}
\end{figure}

The only remaining case is when the only degree-3 neighbor of $u$ is $v$ and there is no other 3-3 edge in the graph. Then the second descendant of $u$ has degree two (by Theorem \ref{thm:1.1}) and we have the situation that is depicted in Figure \ref{fig:R3}. The suggested change improves the ABC-index when $d_R\geq5$. We showed above that we cannot have two 3-3 edges, therefore $d_R\neq 3$ and we may assume that $d_R=4$. If $R$ is the root, then again we have a small tree and, as mentioned before, small trees are known to satisfy our theorem (see \cite{comp1100}). We may assume that there is a vertex $w$ with degree $>4$ (if $d_w=4$, then we should have selected the edge $wR$ for our process). Since $d_R>d_u$, descendants of $R$ should have degree at least 3, otherwise we could have changed them with the $B_2$ branch and get a smaller ABC-index. Figure \ref{fig:R4} presents this case and the suggested structure improves the ABC-index for $d_w\geq 5$. This completes the proof.
\end{proof}

Combining Theorems \ref{thm:1.3} and \ref{thm:2.1}, we get the following corollaries.

\begin{corollary}
In any ABC-minimal tree with maximum vertex degree $\Delta$, there are at most two vertices whose degrees are equal to $\Delta$, and if there are two, they are adjacent.
\end{corollary}

\begin{corollary}
\label{cor:degrees decrease}
In any ABC-minimal tree with maximum vertex degree $\Delta$, the degree sequence on any path starting from a vertex of maximum degree is strictly decreasing with the following exceptions:
\begin{itemize}
\item When two consecutive vertices on the path have degree 2 and the tree is less than 415 vertices.
\item When the path starts with two vertices whose degrees are equal to $\Delta$.
\end{itemize}
\end{corollary}

Another consequence of these results is that the property (P1) holds also when $v$ is a successor of $u$.

\begin{itemize}
  \item[\rm (P1')] Let $T$ be an ABC-extremal tree. Suppose that $u,v\in V(T)$ and $h(u)<h(v)$. Then $T_{u}\succeq T_{v}$ and, in particular, $d_{u}\ge d_{v}$.
\end{itemize}

\begin{proof}
If $v$ is not a successor of $u$ or when $u$ is the root, the property is just (P1). Therefore we may assume that $u$ is not the root and $T_v\subset T_u$. By Corollary \ref{cor:degrees decrease}, we have that $d_u \ge d_v$. Suppose, for a contradiction, that $T_u \prec T_v$. Then we have $d_u\le d_v$, which implies that $d_u = d_v$. Applying Corollary \ref{cor:degrees decrease} again, we conclude that $v$ must be a son of $u$ and that their degree is either 2 or $\Delta$. Clearly, $d_u=d_v=2$ gives that $T_u \succ T_v$; in the other case, $u$ must be the root, a contradiction.
\end{proof}

\section{Vertices at distance 2 from the root}
\label{sect:4}

In this section we will show that vertices at distance at least 2 from the root have degree at most 5. Combining this with the results in the previous section, we will be able to conclude that the diameter of ABC-minimal trees is bounded. Along the way we will prove several other properties of ABC-extremal trees.

\begin{theorem}\label{main}
In any ABC-extremal tree, every vertex of degree at least 6 is either the root or is adjacent to the root.
\end{theorem}

The proof of Theorem \ref{main} consists of two parts. First, we prove a weaker statement (Lemma \ref{lem:distance2fromR} below) which has a possible exception when the result may not hold. In the rest of the section we shall then prove that such an anomaly does not occur (Lemma \ref{lem:noU-exceptional}).

\subsection{Exceptional branches}

To evolve terminology, let us say that a son $u$ of the root is a \emph{$U$-exceptional vertex} if it has a son of degree at least 6; its subtree $T_u$ is said to be a \emph{$U$-exceptional branch}.

\begin{lemma}
\label{lem:distance2fromR}
Let\/ $T$ be an ABC-extremal tree. Then $T$ has at most one $U$-exceptional vertex, and if $u$ is such a vertex, then $u$ has largest degree among the sons of the root and $u$ also has a neighbor of degree at most\/~$5$.
\end{lemma}

\begin{proof}
Let $R$ be the root of $T$ and $u$ be the neighbor of $R$ with the $\succ$-largest subtree $T_u$. Note that this implies that $u$ is a son of $R$ with maximum degree. We may assume that $d_u > 2$. We will contract the edge $Ru$ and add a neighbor to some vertex $w$ of degree one instead. By Theorem \ref{thm:1.1}, the neighbor of $w$ has degree 2 and is thus different from $u$. The change is shown in Figure \ref{fig:dis2}, where $y_i$ $(i=1, \ldots, d_u-1)$ are sons of $u$ and $z_j$ $(j=1, \ldots, d_R-1)$ are the neighbors of $R$ different from $u$.

\begin{figure}[htb]
    \centering
    \includegraphics[width=0.67\textwidth]{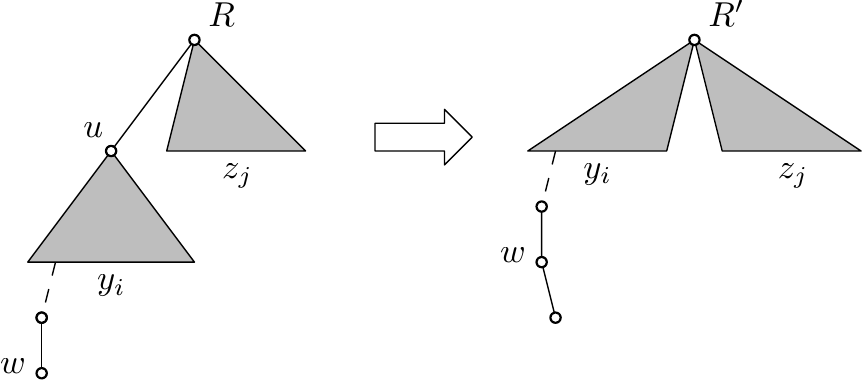}
    \caption{Suggested change for the proof of Lemma \ref{lem:distance2fromR}.}
    \label{fig:dis2}
\end{figure}

Note that $d_R\geq d_u$, $d_{z_j}\geq d_{y_i}$ ($j=1, \ldots, d_R-1$; $i=1, \ldots, d_u-1$). As in our earlier proofs, let $\alpha(T)$ ($\alpha(T')$) be the sum of $f$-values of those edges in $T$ ($T'$) whose contribution to the ABC-index has changed:
$$\alpha(T )=f(2,1)+f(d_R,d_u)+\sum_{j=1}^{d_R-1} f(d_R,d_{z_j})+\sum_{i=1}^{d_u-1} f(d_u,d_{y_i})$$
and
$$\alpha(T' )=f(2,1)+f(2,2)+\sum_{j=1}^{d_R-1} f(d_R+d_u-2,d_{z_j})+\sum_{i=1}^{d_u-1} f(d_R+d_u-2,d_{y_i}).$$

We consider a vertex $z\in \{z_j \ |\  j=1, \ldots, d_R-1\}$ such that $f(d_R,d_{z_j}) - f(d_R+d_u-2,d_{z_j}) \geq f(d_R,d_z) - f(d_R+d_u-2,d_z)$ for all $j=1, \ldots, d_R-1$. Considering a similar inequality for the sons of $u$ (and denoting by $y$ the corresponding vertex where the minimum is attained), we have:
\begin{align}
\alpha(T) - \alpha(T') \geq
  & ~f(d_R,d_u) - f(2,2) +(d_R-1)(f(d_R,d_z) - f(d_R+d_u-2,d_z)) + \nonumber\\
  & ~(d_u-1)(f(d_u,d_y) - f(d_R+d_u-2,d_y)).\label{eq:5*}
\end{align}

Proposition \ref{prop:2} shows that the two differences within the parentheses in (\ref{eq:5*}) are increasing in terms of parameters $d_z$ and $d_y$. Because it suffices to consider the worst case, we may consider their smallest value. If $d_z,d_y\ge6$, we take the value 6. Then (\ref{eq:5*}) changes to:
\begin{align}
\alpha(T) - \alpha(T') \geq
  & f(d_R,d_u) - f(2,2) +(d_R-1)(f(d_R,6) - f(d_R+d_u-2,6)) + \nonumber\\
  & (d_u-1)(f(d_u,6) - f(d_R+d_u-2,6)).\label{eq:6*}
\end{align}
For every $d_R\ge 100$, the value on the right-hand side of (\ref{eq:6*}) is decreasing and its smallest value is when $d_u=d_R$.\footnote{The same holds when $d_R<100$, although the function is not always decreasing.} By considering the values when $d_u=d_R$, we see that the values decrease when $d_R$ grows and that the values are always positive.\footnote{The lower bound becomes $\sqrt{6}/3 - \sqrt{2}/2 > 0$ in the limit when $d_R\to\infty$.} This gives a contradiction when $d_R\ge100$.
On the other hand, we have calculated the lower bound in (\ref{eq:6*}) for all values $d_u\le d_R\leq 100$ by computer and it turns out that we always have $\Delta(T,T')>0$. We conclude that $\abc(T) > \abc(T')$ for all values of $d_R$. This contradiction completes the proof when $d_z,d_y\ge6$.

Now, let us consider the case where the degrees of some of $y_i$ or $z_j$ are less than $6$ and degrees of some of them are $\geq6$.
By property (P1) of ABC-extremal trees, we see that $d_{z_j}\geq d_{y_i}$ for all $j\in\{1, \ldots, d_R-1\}$ and $i\in\{1, \ldots, d_u-1\}$.

Suppose first that $d_{z_1}<6$. Then $d_{y_i}<6$ for all $i\in\{1, \ldots, d_u-1\}$.
Next, consider any other neighbor $z_j$ of $R$. Recall that $T_u\succeq T_{z_j}$. If $d_u=d_{z_j}$, property (P3) implies that all sons of $z_j$ have degree at most 5. If $d_u>d_{z_j}$, the same conclusion follows by Lemma \ref{lem:similar}(a).
This shows that all vertices at distance 2 or more from $R$ have degree at most 5.

We may now assume that $d_{z_j}\geq6$ for all $j\in\{1, \ldots, d_R-1\}$ and that $y_1$ has degree less than $6$. In the same way as above, we see that all sons of $z_1,\dots,z_{d_R-1}$ have degree $\le5$. Hence, at most one neighbor of the root, namely $u$, may have sons of degree $\ge6$ and of degree $\le5$.

Since $d_{u} > d_{y_i}$ (by Theorem \ref{thm:2.1}) and $u$ has a descendant of degree $\leq 5$, the degrees of all descendants of $y_i$ are less than 6 by (P1). This completes the proof.
\end{proof}

In the remainder of this section we will show that $U$-exceptional branches do not exist in ABC-extremal trees, see Lemma \ref{lem:noU-exceptional}. This will make the proof of Theorem \ref{main} complete.

\begin{lemma} \label{lem:d5}
Every ABC-minimal tree has at most eleven vertices of degree 3 and at most four vertices of degree 5. Moreover, there is at most one $B_5$-branch.
\end{lemma}

\begin{proof}
Theorem \ref{thm:2.1} shows that the only way to have a vertex $v$ of degree 3 is when $T_v$ is a $B_2$ or a $B_2^*$-branch. Note that the 2-2 edge of a possible $B_2^*$-branch can be moved to other branches by a similarity exchange, so we may assume that there is no $B_2^*$ in our ABC-extremal tree $T$ (unless all vertices of degree 1 are within $B_2$-branches and one $B_2^*$-branch).
It is shown in \cite{ref21, ref19, ref20} that in any ABC-minimal tree there are at most 11 $B_2$- or $B_2^*$-branches, so there are at most 11 vertices of degree 3.

Let us assume for a contradiction that we have at least 5 vertices of degree 5 and let $v$ be one of them with $u_i$ ($i=1, 2, 3, 4$) as its descendants. Corollary \ref{cor:degrees decrease} indicates that $2\leq d_{u_i}< 5$. If all of them have degree 2 it means that we have a $B_4$ (since $B_4^*$ does not exist by Lemma \ref{B_k branches are small}) and it is known that at most four $B_4$-branches can exist \cite{ref18}. 
If every $u_i$ has degree 4, the change indicated in Figure \ref{fig:4d4} gives us an ABC-smaller tree, and if all $u_i$ have degree 3 the change shown in Figure \ref{fig:4d3} gives us an ABC-smaller tree.
Note that the sons of different degree-5 vertices can be reshuffled by similarity exchanges (without changing the value of the ABC-index). Thus we could group their sons of degree 4 or 3 together. This implies that the vertices of degree 5 cannot have more than three sons of degree 3 all together, and at most three sons of degree 4 all together. Therefore we have at least 13 sons of degree 2, which gives us 3 copies of $B_4$. If we have more than one vertex of degree 3 or 4 as children of vertices of degree 5, then by similarity exchanges we can get at least 6 copies of $B_1^-$ (sons of degree 2 that are not part of a $B_k$ branches) which is not possible in any ABC-minimal tree (see \cite{ref19}). And if we have at most one vertex of degree $>2$ among the children of vertices of degree 5, then simply detach it from its parent and attach it to the grandparent. It is easy to check that this improves the ABC-index. This completes the proof and shows that we have at most 4 vertices of degree 5.

Suppose now that $T$ contains two copies of $B_5$. We replace one of them with a $B_3^*$ and the other one with two copies of $B_3$ (so the degree of the father of one of $B_5$-branches increases by 1). It is easy to see that this decreases the ABC index. This completes the proof.
\end{proof}

\begin{figure}[htb]
    \centering
    \includegraphics[width=0.75\textwidth]{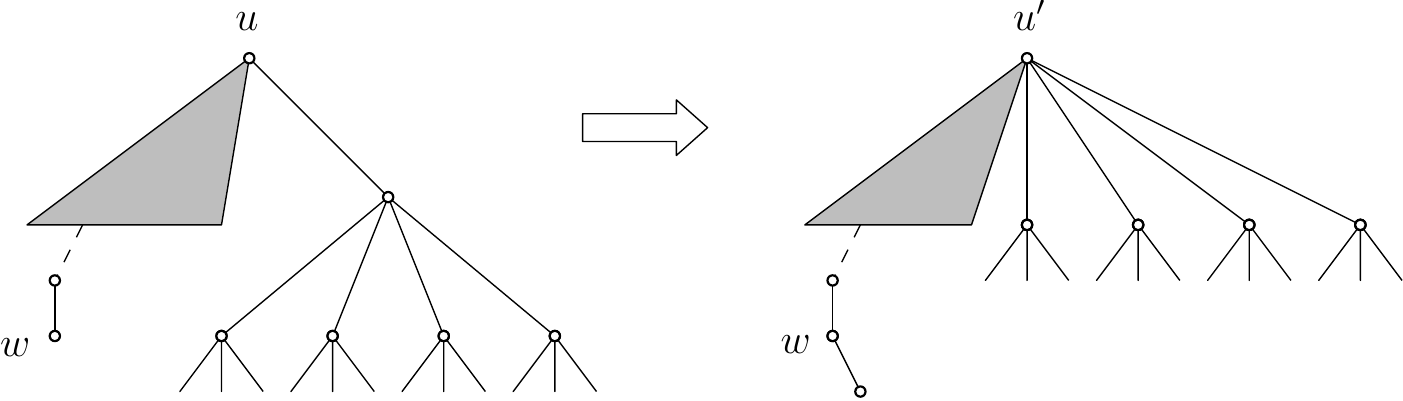}
    \caption{The change of a tree when a 5-vertex has four sons of degree 4.}
    \label{fig:4d4}
\end{figure}

\begin{figure}[htb]
    \centering
    \includegraphics[width=0.75\textwidth]{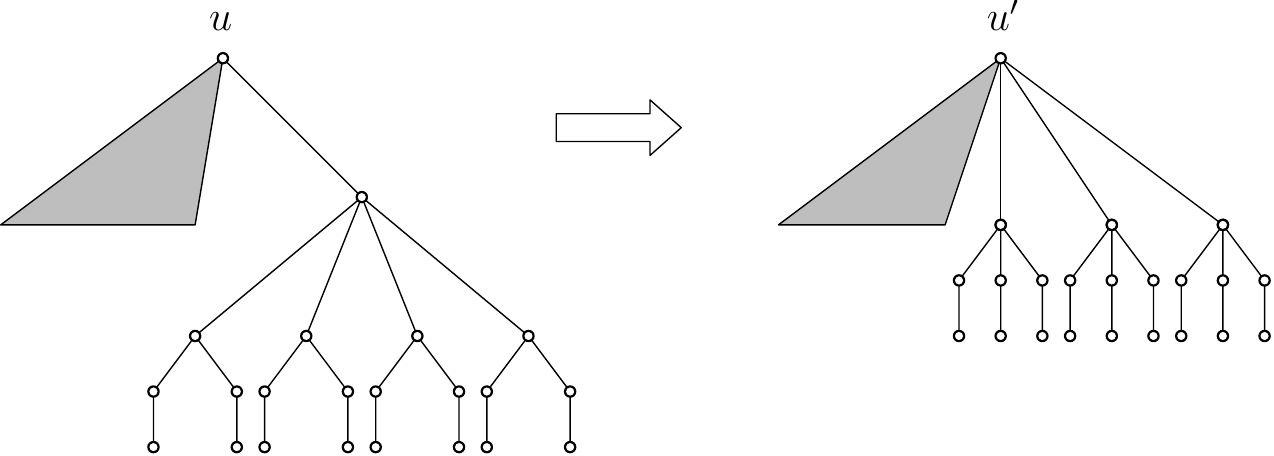}
    \caption{The change of a tree when a 5-vertex has four sons of degree 3.}
    \label{fig:4d3}
\end{figure}

\begin{lemma} \label{lem:deg2root}
If\/ $T$ is an ABC-extremal tree with at least 40 vertices, then no neighbor of the root can have degree~$2$.
\end{lemma}

\begin{proof}
As mentioned before, all ABC-minimal trees with at most 1100 vertices are known \cite{comp1100}. It turns out that the largest one among them having a degree-2 neighbor of the root has 39 vertices. Thus, we may assume that $n>1100$.

Suppose that $R$ has a neighbor $u$ of degree 2. By (P1), all vertices at distance 2 from $R$ have degree 2 or 1. If $T$ has a 2-2 edge, then we may assume that it is incident with $u$ by (P2). Thus, all neighbors of the root form $B_1^-$-branches and $B_k$-branches, where $1\le k\le 5$ (Lemma \ref{B_k branches are small}) which means that we will have a so called Kragujevac tree, see \cite{krag_tree}. Every $B_k$ ($1\le k\le 5$) has at most 11 vertices. Therefore $d_R\ge \lceil \tfrac{1100}{11} \rceil = 100$.

If there are four $B_1^-$-branches, we replace them by one $B_3^*$-branch. Let $x_1,\dots, x_{d_R-4}$ be the other neighbors of $R$. Then we have
\begin{eqnarray*}
   \Delta(T,T') &\ge& f(2,2) - f(4, d_R-3) + \sum_{i=1}^{d_R-4}(f(d_{x_i},d_R) - f(d_{x_i},d_R-3)) \\
   &\ge& f(2,2) - f(4, d_R-3) + (d_R-4)(f(6,d_R) - f(6,d_R-3))
\end{eqnarray*}
which is positive for $d_R\ge 12$, a contradiction.

So we have at most three $B_1^-$-branches. We may have up to four $B_4$ and one $B_5$. Thus, there must exist a $B_k$ for $k\in\{2,3\}$. Now we replace $B_k$ and $B_1^-$ with a $B_{k+1}$. Again, it is easy to see that for $d_R\ge17$, this change decreases the ABC-index, a contradiction.
\end{proof}

Recall that a vertex $u$ is in $B(T)$ if there exist vertices $v$ and $w$ such $u$ is the parent of $v$, $v$ is the parent of $w$, $d_v=2$ and $d_w=1$; and $T_u$ is a $B$-exceptional branch if $u$ also has a son of degree more than 2 or a son of degree 2 that is incident with a 2-2 edge. In the following lemma we will show that in addition to $B_k^*$-branches only one type of $B$-exceptional branches can exist in ABC-minimal trees.

\begin{lemma} \label{no_b_exception}
Each ABC-extremal tree either contains no $B$-exceptional branches, or contains a single $B$-exceptional branch which is isomorphic to $B_2^*$, $B_3^*$, or to the tree $B_3^{**}$ depicted in Figure \ref{B32}.
\end{lemma}

\begin{figure}[htb]
    \centering
    \includegraphics[width=0.2\textwidth]{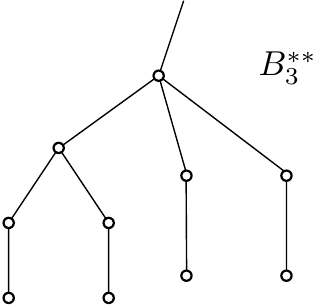}
    \caption{The only possibility for a $B$-exceptional branch different from $B_k^*$-branches.}
    \label{B32}
\end{figure}

\begin{proof}
Suppose first that there is a $B$-exceptional branch that is different from $B_k^*$ ($k\ge2$), and let $u$ be its root. Let $x$ be the child of $u$ with  the smallest degree, subject to the condition that $d_x\geq 3$. By the definition of $B(T)$ we know that $u$ also has a child of degree 2. Let $y_1,\dots, y_{d_u-3}$ be the remaining sons of $u$. By definition of $B(T)$, $u$ is not the root, and thus it has a predecessor $r$.
Consider the change shown in Figure \ref{no_B_excep}.

\begin{figure}[htb]
    \centering
    \includegraphics[width=0.5\textwidth]{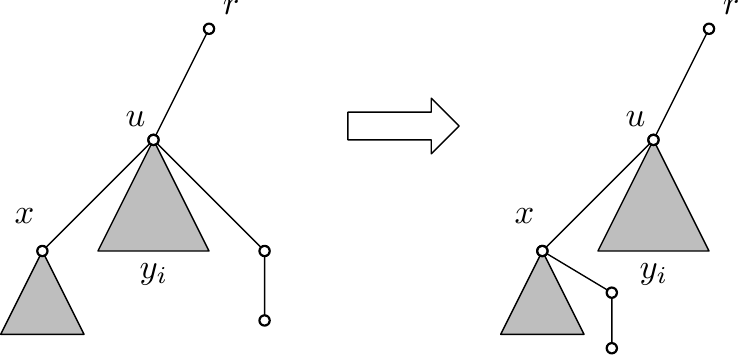}
    \caption{Changing the tree having a $B$-exceptional branch.}
    \label{no_B_excep}
\end{figure}

Since $u$ has a child of degree 2, (P1) implies that all children of $x$ and all children of $y_i$ are of degree $\leq2$. And since vertices of degree one are adjacent only to vertices of degree 2, $T_x$ and each of $T_{y_i}$ is a $B_1^-$ or $B_k$-branch for some $2\le k\leq 4$ (it is easy to see that $B_5$ and $B_1^-$ cannot happen at the same time; replace them with one $B_2$ and one $B_3^*$). Note that for this conclusion we also use (P2).
We will assume that there are no 2-2 edges. If there is one, we first contract it and make the same changes as in the continuation of this proof. After the changes, we uncontract the 2-2 edge. The fact is that the difference $\Delta(T,T')$ will be exactly the same as when there are no 2-2 edges. We will return to this at the end of the proof.

Since the edges incident to vertices of degree 2 have their $f$-value constant, we have:
\begin{eqnarray*}
 \Delta(T,T') & = &  f(d_r,d_u)+f(d_u,d_x)+\sum_{i=1}^{d_u-3}f(d_u,d_{y_i})\\
&& -f(d_r,d_u-1)-f(d_u-1,d_x+1)-\sum_{i=1}^{d_u-3}f(d_u-1,d_{y_i}).\\
\end{eqnarray*}
Using Proposition \ref{prop:2} we see that the above difference is decreasing in $d_r$, $d_x$ and $d_{y_i}$.
Therefore it suffices to show that the difference is positive when replacing $d_r$, $d_x$ and $d_{y_i}$ by largest values that are allowed for these degrees. We can replace $f(d_r,d_u)-f(d_r,d_u-1)$ with the limit when $d_r$ tends to infinity, which is equal to $\sqrt{1/d_u}-\sqrt{1/(d_u-1)}$. If $d_x=5$, then $u$ has at most three descendants $y_i$ whose degree is more than 2, since they would all be of degree 5 by our choice of $x$. In this case we would have:
$$
   \Delta(T,T') > \frac{1}{\sqrt{d_u}} + 4f(d_u,5) -\frac{1}{\sqrt{d_u-1}}-3f(d_u-1,5)-f(d_u-1,6)\\
$$
which is positive for every $d_u\geq13$.

On the other hand, if $d_x\le4$, then we similarly have:
$$
  \Delta(T,T') > \frac{1}{\sqrt{d_u}}+4f(d_u,5)+(d_u-6)f(d_u,4) -\frac{1}{\sqrt{d_u-1}}-5f(d_u-1,5)-(d_u-7)f(d_u-1,4)\\
$$
which is easily seen to be positive for every $d_u\geq 15$. Thus, we may assume from now on that $d_u\le 14$.

As mentioned above, the children of $u$ are copies of $B_1^-$ and $B_i$-branches for $i=2,3,4$. Let $k_i$ be the number of children whose subtrees are isomorphic to $B_i$ (for $i=2,3,4$) or $B_1^-$ (for $i=1$). We know that $4\leq d_u\leq 14$, $1\leq k_1\leq d_u-2$,
$0\leq k_2\leq \min\{d_u-2,11\}$, $0\leq k_3\leq d_u-2$ and $0\leq k_4\leq \min\{d_u-2,4\}$.
Also, either $k_2$ or $k_4$  is zero; moreover, $k_2$, $k_3$ and $k_4$ cannot all be zero.\footnote{There are additional restrictions that follow from (P1), but they are not needed for the proof.} Let ${\mathcal K}$ be the set of all 4-tuples $(k_1,k_2,k_3,k_4)$ satisfying all these conditions with $d_u=k_1+k_2+k_3+k_4$, and let
$$
   k = 1 + 2k_1 + 5k_2 + 7k_3 + 9k_4
$$
be the number of vertices in the $B$-exceptional branch. 
Note that we will have at least one $B_1^-$ and at least one $B_i$ branch and therefore $k\geq8$. Also $k=8$ cannot happen in the ABC-minimal tree because we will have a 3-3 edge. The only way to get $k=10$ (without contradicting the previously mentioned properties) is the $B_3^{**}$ branch which we are claiming to be the only B-exceptional branch in the absence of a 2-2 edge. Observe that $k=9$ or $11$ is not possible, thus we have $k\ge12$. For every possible $(k_1,k_2,k_3,k_4) \in {\mathcal K}$, we replace the $k$ vertices in the $B$-exceptional branch with copies of $B_3$-branches attached to $R$ as follows:

\begin{figure}[htb]
    \centering
    \includegraphics[width=0.65\textwidth]{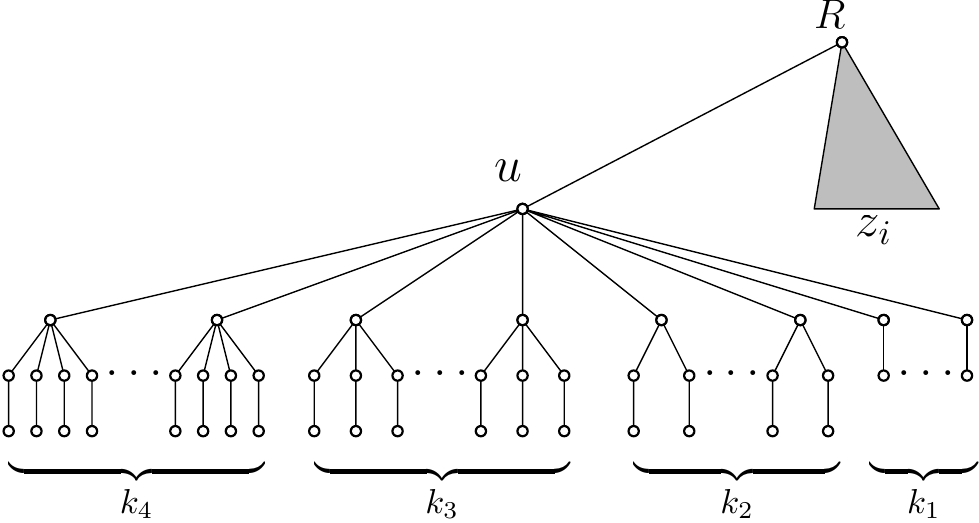}
    \caption{A tree having a $B$-exceptional branch when $d_u\leq 14$.}
    \label{b_excep}
\end{figure}

\begin{itemize}
\item If $k\equiv 0 \mod 7$, replace them with $\frac{k}{7}$ copies of $B_3$.
\item If $k\equiv 1 \mod 7$, replace them with $\frac{k-1}{7}-1$ copies of $B_3$ and one copy of $B_3^*$.
\item If $k\equiv 2 \mod 7$, replace them with $\frac{k-2}{7}-1$ copies of $B_3$ and one copy of $B_4$.
\item If $k\equiv 3 \mod 7$, replace them with $\frac{k-3}{7}-1$ copies of $B_3$ and one copy of $B_3^{**}$.
\item If $k\equiv 4 \mod 7$, replace them with $\frac{k-4}{7}-2$ copies of $B_3$ and two copies of $B_4$.
\item If $k\equiv 5 \mod 7$, replace them with $\frac{k-5}{7}$ copies of $B_3$ and one copy of $B_2$.
\item If $k\equiv 6 \mod 7$: For all cases except the one depicted in Figure \ref{b_excep2} replace them with $\frac{k-6}{7}$ copies of $B_3$ and one copy of $B_2^*$. For the remaining case, when $d_R\leq 94$ we can again replace them with two copies of $B_3$ and one copy of $B_2^*$, but for larger $d_R$ we replace them with one copy of $B_4$ and one copy of $B_5$.
\end{itemize}
\begin{figure}[htb]
    \centering
    \includegraphics[width=0.25\textwidth]{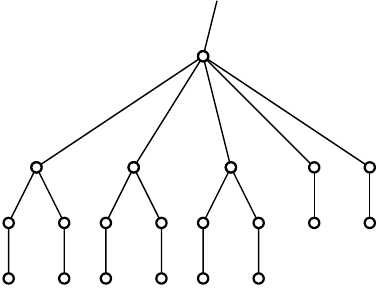}
    \caption{The $B$-exceptional branch that needs special treatment for $d_R\geq 95$.}
    \label{b_excep2}
\end{figure}

A simple verification using computer shows that in all cases, where $k\geq 12$ and $(k_1,k_2,k_3,k_4) \in {\mathcal K}$, the degree of $R$ increases and therefore the change in the ABC-index is increasing in term of $d_{z_i}$. Thus, to consider the worst case, we let $d_{z_i}=3$ (for $i=1,\ldots,d_R-1$). Now the change only depend on one variable ($d_R$) and it is easy to check by computer that the change is possible and it improves the ABC-index. Therefore the only $B$-exceptional branch different from $B_k^*$ for $k\ge2$ that can occur in ABC-extremal trees is $B_3^{**}$.

Let us now return to the case when we had a 2-2 edge. As mentioned above, the same proof works and we conclude that all ending branches (those based at vertices in $B(T)$) are $B_2$, $B_3$, $B_4$, $B_5$, and $B_3^{**}$. Now we put the 2-2 edge back. It can be added to any degree-2 vertex. By Lemma \ref{B_k branches are small}, there are no $B_k^*$ for $k=4,5$. If there is no $B_3^{**}$, then we obtain a single $B_2^*$ or $B_3^*$, as claimed. On the other hand, if $B_3^{**}$ (with its root $u$) is present, then we uncontract the 2-2 edge within this branch and then replace the whole branch with a $B_5$. It is easy to see that this change decreases the ABC-index (which is a contradiction) if the degree of the father $r$ of $u$ is at least 9. Thus, we may assume that $5\le d_r\le 8$.
In this case we can replace $B_3^{**}$ together with the expanded 2-2 edge by $B_2$ and $B_2^*$ attached to $r$. Then we have
\begin{eqnarray*}
   \Delta(T,T') &\ge& f(3,4) + f(4,d_R) - 2f(3, d_R+1) + \sum_{i=1}^{d_R-1}(f(d_{x_i},d_R) - f(d_{x_i},d_R+1)) \\
   &\ge& f(3,4) + f(4,d_R) - 2f(3, d_R+1) + (d_R-1)(f(3,d_R) - f(3,d_R+1))
\end{eqnarray*}
which is positive for $d_r\in \{5,6,7,8\}$.
\end{proof}

Let us observe that $B_3^{**}$ cannot be excluded in all cases. In fact, the tree in Figure \ref{fig:B3starstar} with $k=43$ $B_3$-branches and one $B_3^{**}$ is an ABC-minimal tree that contains $B_3^{**}$. It has $n=312$ vertices.
This is the smallest ABC-minimal tree containing a $B_3^{**}$, see \cite{ref25}.
However, we believe that $B_3^{**}$ cannot occur when $n$ is sufficiently large.

\begin{figure}[htb]
    \centering
    \includegraphics[width=0.32\textwidth]{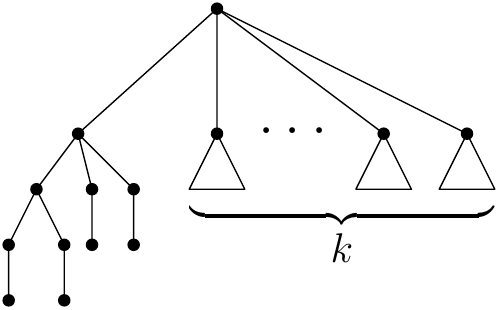}
    \caption{An ABC-minimal tree containing a $B_3^{**}$-branch.}
    \label{fig:B3starstar}
\end{figure}

By Lemmas \ref{lem:B} and \ref{no_b_exception} the following corollary is immediate.

\begin{corollary}\label{cor:endingbranches}
Let\/ $T$ be an ABC-extremal tree and $u$ be a non-root vertex with $d_u\ge3$ and with a son of degree 2. Then $T_u$ is isomorphic to one of the following: $B_2$, $B_2^*$, $B_3$, $B_3^*$, $B_3^{**}$, $B_4$, or $B_5$. Any $B$-exceptional branch can occur at most once, $B_2$ can occur at most eleven times, $B_4$ can occur up to four times, and $B_5$ can occur at most once.
\end{corollary}

Note that if there is $B_2^*$ or $B_3^*$, then it is just one of them and there cannot be any of $B_4$, or $B_5$, because in that case the 2-2 edge could be used to make a $B_4^*$ or $B_5^*$, contrary to Lemma \ref{B_k branches are small}.  Also, we cannot have a copy of $B_3^{**}$ together with a $B_4$ or a $B_5$ since this would contradict (P1) ($B_3^{**}$ contains a 4-3 edge and $B_4$ and $B_5$ have a 5-2 and 6-2 edge, respectively).


\subsection{$C_k$-branches}

\begin{lemma}\label{lem:k_is_51}
If\/ $T$ is an ABC-minimal tree with a vertex $u$ adjacent to (at least) $365$ roots of $C_k$-branches, then $k\leq52$.
\end{lemma}

\begin{proof}
Suppose that there are $365 = 7\times 52 + 1$ copies of $C_k$ adjacent to $u$. Let us consider the remaining neighbors of $u$, and if there are any, let $R$ be one of them that has the highest degree. (If $u$ is not the root then $R$ is the parent of $u$). We can apply the change shown in Figure \ref{fig:k_is_51} to obtain a tree $T'$. Note that the two trees have the same number of vertices and that $d_{u'} = d_u+7k-364$.

\begin{figure}[htb]
    \centering
    \includegraphics[width=0.65\textwidth]{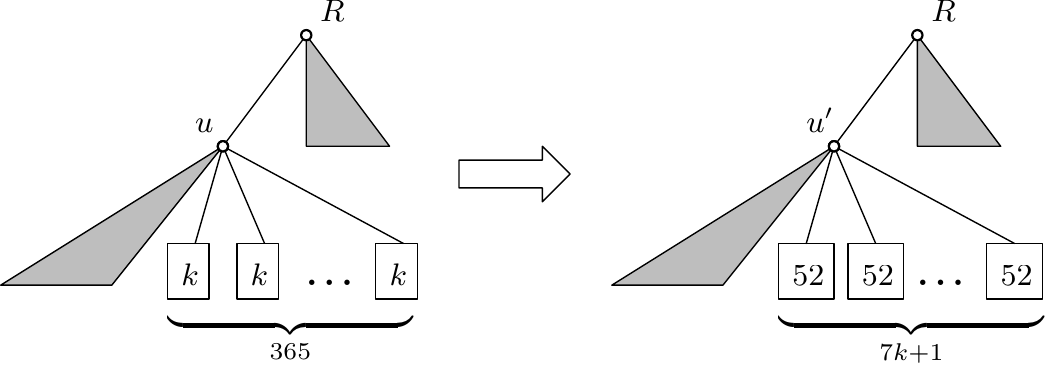}
    \caption{Suggested change when there are 365 copies of $C_k$.}
    \label{fig:k_is_51}
\end{figure}

Let $d=d_u-366$ and if $d\ne -1$, let $a_1,\dots, a_d$ be the degrees of the sons of $u$ in the shaded part. If $k>52$, we have $d_u<d_{u'}$ and by Proposition \ref{prop:2}, the differences 
$f(R,u) - f(R,u')$ and $f(u,a_i)-f(u',a_i)$ are increasing in terms of $d_R$ and each $a_i$. Suppose first that $R$ exists. Then it suffices to consider the case when $u$ is the root and $d_R=4$ and each $a_i=4$ (since $d_u>k+1$). Then we have:
\begin{eqnarray}
\Delta(T,T') &\ge& 365 \Bigl(f(d_u,k+1 )+k\Bigl(f(k+1,4) +6\tfrac{\sqrt{2}}{2}\Bigr)\Bigr) + (d+1)\, f(d_u,4) \nonumber \\
 && -(7k+1) \Bigl( f(d_{u'},53)+52\Bigl(f(53,4) + 6\tfrac{\sqrt{2}}{2}\Bigr)\Bigr) - (d+1)\,f(d_{u'},4). \label{eq:52}
\end{eqnarray}


Note that by Lemma \ref{lem:C_k}, $k<143$. For a fixed value of $k$ the equation only depends on $d_u$ and it is easy check by computer that the suggested change decreases the ABC-index for all $53\geq k \geq 142$ and $d_u\geq 365$.

If $R$ does not exist, then $d=-1$ and $\Delta(T,T')$ has the same lower bound (\ref{eq:52}). From this we obtain the same conclusion. This completes the proof.
\end{proof}

\begin{lemma}\label{lem:7k+8}
Let $u$ be a vertex in an ABC-minimal tree and let $k$ be a positive integer. If $k\leq 48$, then there are at most $7k+7$ copies of $C_k$-branches whose roots are the sons of $u$. If $d_u$ is greater than $474$ $(874$ and\/ $3273$, respectively$)$, then there are at most $7k+7$ copies of $C_k$ for $k=49$ $(k=50$ and $k=51$, respectively$)$.
\end{lemma}

\begin{proof}
Let $u$ be a vertex that has $7k+8$ copies of $C_k$ as his children.
Let $x_1,\dots,x_m$ be the sons of $u$ that are not in the considered $C_k$-branches.
First we will discuss the degree of the vertices $x_i$. Let $\mathcal{I} = \{i : d_{x_i}>k+2\}$. As discussed at the end of Lemma \ref{lem:distance2fromR} when $u$ is not the root, the degree of all grandchildren of $u$ is at most 5. Since $d_{x_i}>k+2$ for $i\in \mathcal{I}$ and there are $B_3$-branches attached to vertices of degree $k+1$, then $x_i$ cannot have children of degree less than 4 (by Lemma \ref{lem:similar}(a). Therefore each $x_i$ ($i\in \mathcal{I}$) has children of degree 4 or 5 only. Recall that (by Lemma \ref{lem:d5}) we have at most 4 vertices of degree 5, and therefore all but at most one of the $x_i$ are roots of $C_{k'}$-branches (after possible similarity exchanges). Lemma \ref{lemma:C} shows that $k'\leq k+1$ and therefore $|\mathcal{I}|\leq1$.

Also note that when $u$ is the root, there is at most one $U$-exceptional branch (we can assume that it is $R$) and using the same argument as above, the degree of all but at most one of $x_i$'s is at most $k+2$.

\begin{figure}[htb]
    \centering
    \includegraphics[width=0.65\textwidth]{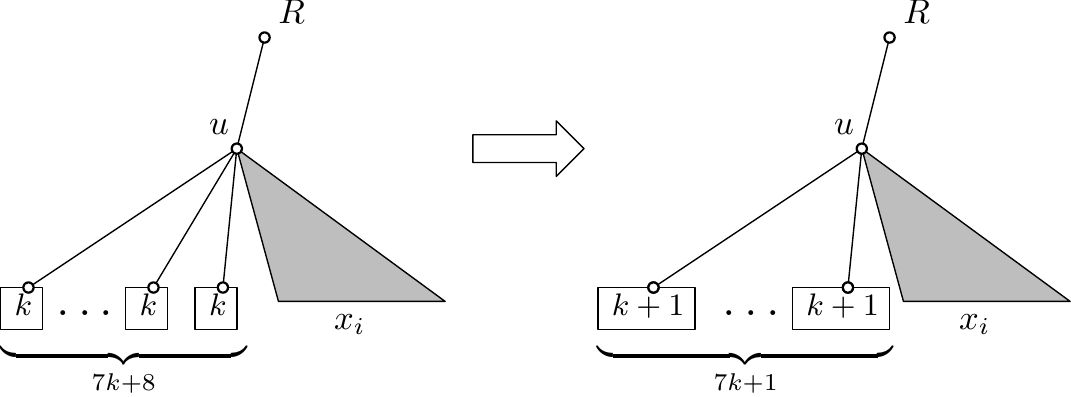}
    \caption{Compactifying $C_k$-branches when $k\leq 51$.}
    \label{number_of_C_k}
\end{figure}

Consider the change depicted in Figure \ref{number_of_C_k}.
The difference between the ABC-indices of the two trees is:
\begin{eqnarray*}
 \Delta(T,T') & = & f(d_R,d_u)-f(d_R,d_u-7)+(7k+8)f(d_u,k+1)-(7k+1)f(d_u-7,k+2)\\
&& +\sum_{i=1}^{d_u-7k-9}f(d_u,d_{x_i})-\sum_{i=1}^{d_u-7k-9}f(d_u-7,d_{x_i})\\
&& + k(7k+8) f(k+1,4)-(7k+1)(k+1)f(k+2,4)-6f(2,1).\\
\end{eqnarray*}
By Proposition \ref{prop:2} this difference is decreasing in $d_R$ and $d_{x_i}$.

Therefore it suffices to consider the limit when $d_R\to\infty$ and to let $d_{x_1}=d_u-1$ and $d_{x_i}=k+2$ for $i=2,\ldots,d_u-7k-9$ (when $u$ is the root the worst possible case for $d_R$ is $d_u-1$ but since the equation is decreasing in $d_R$, considering $d_R\to \infty$ is enough and covers other cases as well). We have:
\begin{eqnarray*}
 \Delta(T,T') & \geq & \frac{1}{\sqrt{d_u}}-\frac{1}{\sqrt{d_u-7}}+(7k+8)f(d_u,k+1)-(7k+1)f(d_u-7,k+2)\\
&& +(d_u-7k-10)\left(f(d_u,k+2)-f(d_u-7,k+2)\right)+ f(d_u,d_u-1)-f(d_u-7,d_u-1)\\
&& + k(7k+8) f(k+1,4)-(7k+1)(k+1)f(k+2,4)-6f(2,1).
\end{eqnarray*}
 Since for a fixed value of $k$ ($\leq 48$) the right hand side of this inequality only depends on $d_u$ ($\geq 7k+8$), one can check that $\Delta(T,T') > 0$ for all values of $d_u$ and for all $k\leq 48$.

Note that this change also works for $k=49,50$ and 51 if $d_u$ is at least 474, 874 and 3273, respectively.
\end{proof}

\subsection{There are no $U$-exceptional branches}

\begin{lemma}
\label{lem:noU-exceptional}
ABC-extremal trees have no $U$-exceptional branches.
\end{lemma}

\begin{proof}
Let $u$ be the root of a $U$-exceptional branch and let $w$ be the child of $u$ with the highest degree ($d_w\geq 6$). Let $x_l$ ($l=1,\dots,m$) be the other children of  $u$ with $d_{x_l}\geq 6$ and let $y_j$ be children  of $u$ with $d_{y_j}\leq 5$ ($j=1, \ldots, d_u-m-2$). Consider the change outlined in Figure \ref{U_exceptional}.

\begin{figure}[htb]
    \centering
    \includegraphics[width=0.82\textwidth]{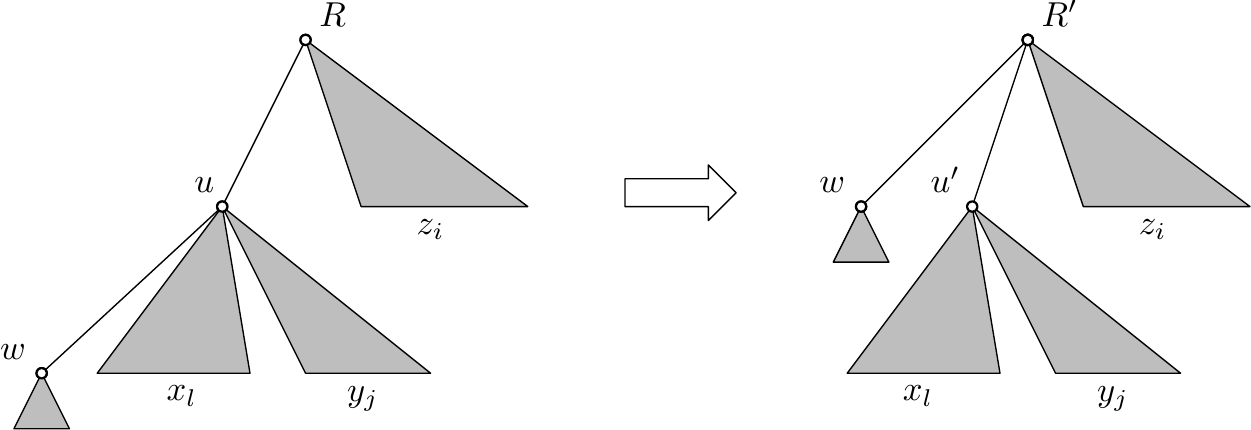}
    \caption{Suggested change when there is a $U$-exceptional branch.}
    \label{U_exceptional}
\end{figure}

First note that since $d_{y_j}\leq 5$ and $d_u\geq d_{z_i}\geq d_{x_l}$, all children of $w$, of each $x_l$ and of each $z_i$ have degree at most 5 and since we have at most 4 vertices of degree 5 and 11 vertices of degree 3 and at most one $B$-exceptional branch, all but at most two subtrees of $x_l$ and at most two subtrees of $z_i$ are $C_k$-branches.
To see this, observe that vertices of degree 3 are roots of $B_2$ or $B_2^*$ branches and to have a vertex of degree 4 (which is not $B_3$, $B_3^*$ or $B_3^{**}$) we should use $B_2$ branches as children of a vertex of degree 4. Since we have at most 11 $B_2$ branches, all but at most 4 vertices of degree 4 are roots of $B_3$ branches (this can be improved to 1, see Lemma \ref{degree k}). Therefore at most two $z_i$ branch will contain vertices of degree 5 (that are not among $y_j$'s) and vertices of degree 4 that are not $B_3$ branches. Also at most two $x_l$ branches will contain any remaining $B_2$ branches.

The difference between the ABC-indices of trees in Figure \ref{U_exceptional} is:
\begin{eqnarray}
\Delta(T,T') & = &  f(d_R,d_u)+\sum_{i=1}^{d_R-1}f(d_R,d_{z_i})+f(d_u,d_w)+
\sum_{j=1}^{d_u-m-2}f(d_u,d_{y_j})+\sum_{l=1}^{m}f(d_u,d_{x_l}) \nonumber\\
&& -f(d_R+1,d_u-1)-\sum_{i=1}^{d_R-1}f(d_R+1,d_{z_i})-f(d_R+1,d_w) \label{eq:changeUexceptional}\\
&& -\sum_{j=1}^{d_u-m-2}f(d_u-1,d_{y_j})-\sum_{l=1}^{m}f(d_u-1,d_{x_l}). \nonumber
\end{eqnarray}
Using Proposition \ref{prop:2} we can see that this difference is increasing in $d_{z_i}$ and $d_w$ and decreasing in $d_{y_j}$ and $d_{x_l}$. In order to verify that the difference is positive, it suffices to prove it when $d_{y_j}=5$ and $d_{z_i}=d_{x_l}=d_w=:d'$ for $i=1,\ldots, d_R-1$, $j=1, \ldots,d_u-m-2$ and $l=1,\ldots, m$. We have:
\begin{eqnarray*}
\Delta(T,T') & \geq &  f(d_R,d_u)-f(d_R+1,d_u-1)+(d_R-1)\left( f(d_R,d')-f(d_R+1,d')\right)\\
&&+f(d_u,d')-f(d_R+1,d')+(d_u-2)\left(f(d_u,5)-f(d_u-1,5)\right)\\
&&+m\left( f(d_u,d') -f(d_u-1,d') - f(d_u,5)+f(d_u-1,5)\right).
\end{eqnarray*}

Note that the coefficient of $m$ is a negative number by Proposition \ref{prop:2} and therefore the right-hand-side of this inequality is decreasing in $m$. Therefore to consider the worst case, we will let $m=d_u-3$ (the highest possible value of $m$ by the definition of $U$-exceptional branches). We have:
\begin{eqnarray*}
\Delta(T,T') & \geq &  f(d_R,d_u)-f(d_R+1,d_u-1)+(d_R-1)\left( f(d_R,d')-f(d_R+1,d')\right)\\
&&+f(d_u,d')-f(d_R+1,d')+\left(f(d_u,5)-f(d_u-1,5)\right)\\
&&+(d_u-3)\left( f(d_u,d') -f(d_u-1,d')\right).
\end{eqnarray*}
Let $g$ be the value of the right-hand side of this inequality. Lemmas \ref{B_k_equal}, \ref{lem:7k+8} and \ref{lem:k_is_51} show that when $d_R\geq 3273$, then the $C_k$-branches adjacent to $R$ are $C_{52}$, with at most 364 exceptions of $C_{51}$ or $C_{53}$ (if we have $d_u\geq 3273$, then the same holds for the branches rooted at the vertices $x_l$). Note that if we have some vertices of degree 54, then  by Lemma \ref{lem:similar} they are among $z_i$ and since the right-side value in (\ref{eq:changeUexceptional}) is increasing in each $d_{z_i}$, the worst case occurs when we have no vertices of degree 54. Also if we have vertices of degree 52, then they are among $x_l$ and since $g$ is decreasing in terms of $d_{x_l}$, the worst case happens when we have no vertices of degree 52. 
Observe that even if $m$ or $d_u$ is small and we do not have enough $C_k$ branches as children of $u$, since the worst case for the degree of each $z_i$ is 53 and because $d_{x_l}\leq d_{z_i}$ and the equation is decreasing in terms of $d_{x_l}$, the worst case for the degree of each $x_l$ is 53.

A simple computer search shows that when $d' \leq d_u<d_R<3272$, $g$ is positive. So the only remaining case to be considered is when $d_R\geq 3272$ and therefore $d'=53$ (i.e. when each $x_l$ and each $z_i$ is the root of a $C_{52}$-branch). The Taylor expansion (in terms of $d_R$) of $\frac{\partial g}{\partial d_u}$ shows that for large values of $d_R$, $g$ is decreasing in $d_u$ and therefore it suffices to show that $g>0$ when $d_u=d_R$ (the highest possible value). By substituting $d_u$ with $d_R$, the function will only depend on $d_R$ and it is easy to check that $g>0$. By a simple computer verification we have also checked that $g>0$ for small values of $d_u\leq d_R\leq 10000$ and $d'=53$.

To summarize, in all cases we have $g>0$. This gives a contradiction and completes the proof.
\end{proof}

\begin{corollary}\label{col:4d5}
In every ABC-extremal tree the vertices of degree one are at distance at most 5 from the root.
\end{corollary}

\begin{proof}
We will use the main results proved above. First of all, every path from the root has decreasing degrees with possible exception of the first edge on the path, or when the path contains a 2-2 edge. Recall that there is at most one 2-2 edge in $T$ (Lemma \ref{lem:2-2edge}).  By Theorem \ref{main}, we may assume that vertices at distance 2 from the root have degree 5 or less.
Suppose that a path starting at the root has length 6 or more. Consider degrees of vertices on the path starting at the vertex at distance 2 from $R$. The only possible cases are the following degree sequences: 4-3-2-2-1, 5-4-3-2-2-1, 5-4-3-2-1, 5-4-2-2-1, 5-3-2-2-1. 


We will first prove that $h_2 = \min \{h(v)\mid d_v=2 \} \le 3$, i.e. there is a degree-2 vertex at distance at most 3 from $R$. If there is no vertex of degree 5 at distance 2 from the root (a bad 5-vertex), this is easy to see (similar arguments as in the more complicated case below), and we omit details. So let us suppose that there is a bad 5-vertex $v$. By (P1), all sons of the root have degree at least 5 in this case. Since there are at most 4 vertices of degree 5 (Lemma \ref{lem:d5}), it follows that there are more than 11 vertices at distance 2 from the root that are not of degree 5. Thus one of them leads to a leaf without including a vertex of degree 3 (Lemma \ref{lem:d5}). The predecessor of this leaf is a degree-2 vertex at distance at most 3 from $R$. Therefore $h_2\leq 3$.

By (P2), if there is a 2-2 edge in the ABC-minimal tree, then the height of one of its vertices should be $h_2$, which is a contradiction to the existence of 4-3-2-2-1, 5-4-3-2-2-1, 5-4-2-2-1 and 5-3-2-2-1.

Now observe that since $h_2\leq 3$, then  5-4-3-2-1 cannot also happen because it has a vertex of degree 3 at distant 4 from the root, which is a contradiction. 
\end{proof}

In fact, distance 5 from the root in the last corollary can be reduced to 4 (if there are no $B_2^*$, $B_3^*$ and $B_3^{**}$).
As this is not important for the main structure results, we do not intend to deal with this improvement here.

\subsection{Vertices of intermediate degree}

\begin{lemma}\label{degree k}
Suppose that $T$ is an ABC-minimal tree.
\begin{itemize}
\item[{\rm (a)}] The only non-root vertices of degrees 3, 4 or 5  are roots of $B_2$, $B_2^*$, $B_3$, $B_3^*$, $B_3^{**}$ or $B_4$ branches.
\item[{\rm (b)}] There are no non-root vertices of degree $k$ for $6\leq k \leq 15$.
\end{itemize}
\end{lemma}

\begin{proof}
First note that ABC-minimal trees of order $\leq 1100$ have been determined \cite{Dim13} and this lemma holds for all of them. So we can assume that $n\geq 1100$ and $d_R\geq 5$. We will use Corollary \ref{cor:endingbranches} throughout.

Let $v$ be a non-root vertex of degree $k$. If $k=3$ then $T_v$ is either $B_2$ or $B_2^*$. If $k=4$ and $T_v$ is not $B_3$, $B_3^*$ or $B_3^{**}$, then all children of $v$ are $B_2$ or $B_2^*$ branches ($T_v$ will have 16 or 17 vertices). It is easy to check that by exchanging $T_v$ with one $B_3$ and one $B_4$ (16 vertices) or one $B_3$ and two $B_2$ branches (17 vertices) will improve the ABC-index. 
If $k=5$ and $T_v$ is not a $B_4$, by Lemmas \ref{no_b_exception} and \ref{B_k branches are small} all children of $v$ have degrees 3 or 4. So in this case all subtrees of $T_v$ are $B_2$, $B_2^*$, $B_3$, $B_3^*$ or $B_3^{**}$.  Let $k_2$, $k_2^*$, $k_3$, $k_3^*$ and $k_3^{**}$ represent the number of each of these branches, respectively. Therefore, $k_2+k_2^*+k_3+k_3^*+k_3^{**}=4$. From previous arguments we know that some of these branches cannot occur at the same time and also we know that $k_2^*+ k_3^*+k_3^{**}\leq 1$. Let $n_v = 1+ 5k_2+6k_2^*+7k_3+8k_3^*+10k_3^{**}$ be the number of vertices in $T_v$ ($21\leq n_v\leq32$). We will replace $T_v$ by the subtrees used in the proof of Lemma \ref{no_b_exception} (treating seven different cases of the value of $k$ modulo 7). As an example, when $n_v=21$ ($n_v\equiv 0$ mod 7) we will replace $T_v$ by 3 copies of $B_3$. Note that there are several cases but in all of them the degree of the parent of $v$, say $R$, increases and  therefore the difference between ABC-indices of these two trees (in the worst case) only depends on $d_R$. It is easy to check by computer that the difference is positive for all values of $d_R\geq 5$. This completes the proof of part (a).

When $k=6$, similar arguments as in part (a) can be used to show that  every vertex of degree 6 is a root of $B_5$. We will prove later that $B_5$ can be excluded as well.

For $k\geq 7$, we give a proof iteratively starting with $k=7, 8,$ etc. Then we may assume that children of $v$ have degrees 3, 4, 5 or 6 and all subtrees of $T_v$ are $B_2$, $B_2^*$, $B_3$, $B_3^*$, $B_3^{**}$, $B_4$ or $B_5$ branches. Let $k_2$, $k_2^*$, $k_3$, $k_3^*$, $k_3^{**}$, $k_4$ and $k_5$ represent the number of each of these branches, respectively. Therefore, $k_2+k_2^*+k_3+k_3^*+k_3^{**}+k_4+k_5=k-1$. We also know that some of these branches cannot occur at the same time and also we know that $k_2\leq 11, k_2^* +  k_3^*+ k_3^{**}\leq 1$ and $k_4\leq 4$. Let $n_v = 1+ 5k_2+6k_2^*+7k_3+8k_3^*+10k_3^{**}+9k_4+11k_5$ be the number of vertices in $T_v$. Again, we will replace $T_v$ by the subtrees used in the proof of Lemma \ref{no_b_exception} (with seven different cases for the value of $k$ modulo 7). The difference between ABC-indecies of these two trees (in the worst case) only depends on $d_R$ and it is easy to check\footnote{This was checked by computer. The program and its output are available from the authors.} that the difference is positive for all values of $d_R\geq k$  which completes the proof of part (b) when $k\neq6$.

Let $v$ be the root of a $B_5$ branch and let $w$ be its father. By the above, $d_w\geq16$. If there is at least one $B_2$ or $B_3$ attached to $w$, Lemma \ref{B_k_equal} applies and we are done. Since we can have up to 4 vertices of degree 5 (and $v$ has more than 4 siblings), there is a sibling $u$ of $v$ of degree at least 16 (by the previous paragraph). It is easy to see that there is a $B_2$ (or $B_3$) as a child of $u$. Now apply the following change: detach a $B_1^-$ from $B_5$ (thus changing it to $B_4$) and attach it to this $B_2$ (or $B_3$). The change of ABC-index is increasing in $d_u$ and decreasing in $d_w$. So the worst case is when the degree of $w$ goes to infinity and $d_u=16$ for which it is easy to check that the change of the index is positive. This contradiction shows that there is no $B_5$ and thus completes the proof.
\end{proof}

\begin{corollary}
In any ABC-extermal tree, all $B_4$ branches (if any) are attached to the root.
\end{corollary}

\begin{proof}
Let $u$ be a non-root vertex which has 1,2,3 or 4 $B_4$ branches as its children. By Lemma \ref{degree k}, the degree of $u$ is at least 16 and by Theorem \ref{main}, it is adjacent to the root. We will first prove an upper bound on the degree of $u$ when the degree of the root is large enough. Consider a change similar to Figure \ref{fig:C_k}  in Lemma \ref{lem:C_k} (we have 1, 2, 3 or 4 copies of $B_4$ as well). One can check that in all of these four cases, degree of $u$ is bounded above by 107 when degree of the root is at least 952. 

Now that we have the desired upper bound, consider the following change. Detach all $B_4$ branches from $u$ and attach them to the root. Degree of the root increases and therefore the difference in the ABC-index is increasing in the degree of children of $R$ (except $u$). Therefore for the worst case we can only consider their smallest possible degree (some of them can have degree 5 and the rest will have degree 4). A simple computer search shows that in different cases of small values of $d_R\leq 951$ (and all $d_u\leq d_R$) the suggested change improves the ABC index. Hence we can assume that $d_R\geq 952$  and therefore $d_u\leq 107$. Now for every fixed value of $d_u$ the worst case of the equation only depends of $d_R$ and it is easy to check that this change decreases the ABC-index.
\end{proof}

\begin{corollary}\label{cor:all but one}
Any ABC-minimal tree is similar to a tree where at most one non-root vertex of degree $\geq 6$ exists that is not a root of a $C_k$-branch.
\end{corollary}

\begin{proof}
We may assume $T$ is an ABC-extremal tree. Let $u$  be a non-root vertex which is not the root of a $C_k$ branch and has $d_u\geq6$. Since there are no such vertices of degree $6,\ldots,15$ we have $d_u\geq 16$ and therefore $u$ is a son of the root. 
Since $d_u\geq16$, $u$ has some $B_3$ branches as subtrees (see Corollary \ref{cor:endingbranches}) and by (P1), (P2) and (P3), either $T_u$ contains all $B_2$ branches (and possibly a $B_2^*$) or $T_u$ contains $B_3^*$ or $B_3^{**}$ branches. If there is a $B_2$ (or $B_2^*$) and also one $B_3^*$ or $B_3^{**}$ in the tree, then by a similarity exchange we can move them and make all of them sons of $u$. Therefore, all other vertices at distance two from the root will be $B_3$ branches and all children of the root (with degree at least 6) except $u$ are $C_k$ branches.
\end{proof}

\begin{corollary}\label{cor:new}
In any ABC-minimal tree of order at least $111$, the degree of the root is at least $16$.
\end{corollary}

\begin{proof}
If the root has a child $u$ that is none of $B_2$, $B_2^*$, $B_3$, $B_3^*$, $B_3^{**}$ or $B_4$, then by Lemma \ref{degree k} $d_u\geq 16$ and therefore the degree of the root is also at least 16 and we are done. So we may assume that all children of the root are $B_2$, $B_2^*$, $B_3$, $B_3^*$, $B_3^{**}$ or $B_4$. In \cite{comp1100}, all ABC-minimal trees of order up to 1100 have been determined and the largest one where the degree of the root is less than 16 is a tree with 110 vertices (contains 2 $B_4$ and 13 $B_3$ branches).
\end{proof}

\begin{lemma}
\label{lem:remove35}
Any ABC-minimal tree is similar to a tree $T$ in which one can remove a set of at most\/ $63$ vertices so that the resulting tree has no vertices of degree 3 or 5 and all ending branches are $B_3$. The removed vertices are all from at most four $B_4$-branches adjacent to the root and from at most one subtree corresponding to a son of the root.
\end{lemma}

\begin{proof}
We may assume $T$ is ABC-minimal. By removing at most one vertex, we eliminate possible 2-2 edge. By removing 3 vertices, $B_3^{**}$ can be changed into a $B_3$. Similarly, by removing 2 vertices, a $B_4$ can be changed into a $B_3$. These operations do not change degrees of vertices except for the roots of newly formed $B_3$-branches. Since a 2-2 edge or $B_3^{**}$ cannot coexist with a $B_4$, this step removes at most 8 vertices all together.

By Lemma \ref{degree k} (a), any vertex of degree 4 is a $B_3$-branch (or $B_3^*$ or $B_3^{**}$) and any non-root vertex of degree 5 is a $B_4$ adjacent to the root. Thus, all remaining vertices of degree 3 have fathers of degree at least 16 (by Lemma \ref{degree k} and Corollary \ref{cor:new}). Each such father $w$ has only $B_3$ and $B_2$ as subtrees. By (P1)--(P3), all $B_2$-branches are descendants of a single such vertex $w$. By removing all ($\leq 11$) $B_2$-branches, all ending branches are $B_3$. By Corollary \ref{cor:all but one} we may assume that possible 2-2 edge or $B_3^{**}$ is also a subtree of $T_w$. this complete the proof.
\end{proof}

\section{The main structure}
\label{sect:5}

\begin{lemma}
\label{lem:C_52_exists}
Let $T$ be an ABC-extremal tree whose root $R$ has degree $d_R\geq 2888$. Then there are at least $d_R-327$ $C_{k}$- and $C_{k+1}$-branches attached to the root for some $k\in \{50,51,52\}$.
\end{lemma}

\begin{proof}
Let $x_1,\dots, x_{d_R}$ be the sons of $R$ and $T_i=T_{x_i}$ their subtrees ($i=1,\dots,d_R$). By removing one of these subtrees and at most four $B_4$-branches (Lemma \ref{lem:remove35}), all the remaining subtrees are copies of $B_3$ and copies of $C_{k_i}$. Lemma \ref{lem:C_k} shows that each such $k_i$ is at most $142$.

Suppose first that there are at most 322 $B_3$-branches among them. Then we have at least $d_R-327 > 2561$ $C_{k_i}$-branches. By Lemma \ref{lemma:C}, the values $k_i$ take only two consecutive values. By Lemmas \ref{lem:k_is_51} and \ref{lem:7k+8}, all but at most 364 of these $C$-branches are $C_{51}$ or $C_{52}$.

Suppose now that there are at least $323$ $B_3$-branches adjacent to $R$. Now consider $s=323$ copies of these $B_3$ branches and apply the change outlined in Figure \ref{fig:s_is_zero4}, where we replace these $B_3$-branches with seven $C_{(s-1)/7}$-branches (there may be more $B_3$ branches attached to the root in the shaded part).

\begin{figure}[htb]
    \centering
    \includegraphics[width=0.92\textwidth]{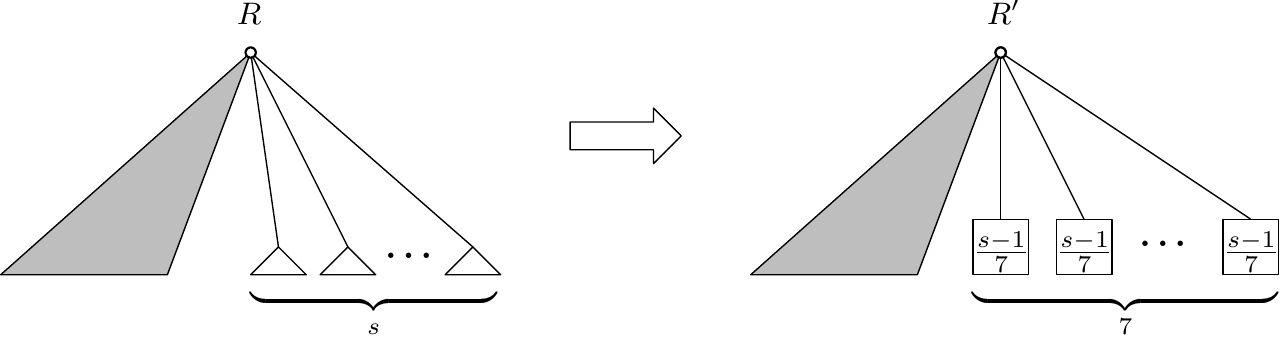}
    \caption{Improving ABC-index if there are $s$ $B_3$-branches, where $s\equiv 1 \pmod 7$.}
    \label{fig:s_is_zero4}
\end{figure}

Let $z_i$ ($i=1,\ldots,d_R-s$) be neighbors of $R$ of in the shaded part of Figure \ref{fig:s_is_zero4}. Then the change of ABC-contribution of the corresponding edge is $f(d_R,d_{z_i})-f(d_R-s+7,d_{z_i})$, which is decreasing in $d_{z_i}$ by Proposition \ref{prop:2}. We know that $d_{z_i}\leq 143$ but it is not possible to have the degree of all of them equal to 143. So we have two extreme cases here. If we have at most $2\times 364$ $C_k$ branches, then we can have 364 of $z_i$'s to be the root of $C_{142}$ and the rest to be the root of $C_{141}$ branches. Note that in this case the other children of the root will be (at most 4) $B_4$ and $B_3$ branches. Therefore the change in the ABC-index will be bounded below as follows:
\begin{eqnarray*}
  \Delta(T,T') &\geq&  364f(d_R,143)+364f(d_R,142)+4f(d_R,5)+(d_R-2\times 364-4)f(d_R,4)+6\sqrt{2}/2\\
  &&  -364f(d_R-s+7,142)-364f(d_R-s+7,143)-7f(d_R-s+7,(s+6)/7)\\
  && -4f(d_R-s+7,5)-(d_R-s-2\times 364-4)f(d_R-s+7,4)-(s-1)f((s+6)/7,4).
\end{eqnarray*}
It can be shown that this is positive for $s=323$ when $d_R\geq2092$.

On the other hand, if we have more than $2\times 364$ $C_k$ branches, then all but $364$ of them are $C_{52}$. So in the worst case the root has 364 children of degree 54 and $d_R-364-s-4$ children of degree 53\footnote{The same can be concluded for the subtrees containing vertices of degree 3 because of (P1)--(P3).} and 4 children of degree 5. Therefore we have:
\begin{eqnarray*}
  \Delta(T,T') &\geq&  364f(d_R,54)+(d_R-364-4)f(d_R,53)+4f(d_R,5)+s\times f(d_R,4)+6\sqrt{2}/2\\
  &&  -364f(d_R-s+7,54)-(d_R-364-4)f(d_R-s+7,53)\\
  && -4f(d_R-s+7,5)-7f(d_R-s+7,(s+6)/7)-(s-1)f((s+6)/7,4).
\end{eqnarray*}
It can be shown that this is positive for $s=323$ when $d_R\geq2888$. This completes the proof.
\end{proof}

\begin{lemma} \label{lem:s_is_zero}
Let $T$ be an ABC-extremal tree of order $n\geq 1078940$. Then there are no $B_3$-branches attached to the root and there are at least\/ $\tfrac{1}{365} n - 377$ $C_{52}$-branches.
\end{lemma}

\begin{proof}
We start as in the previous proof. By removing at most 63 vertices, we end up with a subtree $T_1$ whose main subtrees are all $B_3$ and $C_{k_i}$ ($i=1,\dots,r$), where $k_1\ge k_2\ge \dots \ge k_r$. First we claim that $T$ contains a $C_{51}$ and $C_{52}$-branch. If $d_R\ge 2888$, then Lemma \ref{lem:C_52_exists} shows that there is a large number of $C_{51}$ or $C_{52}$ among the main subtrees, so one of them is also contained in $T$. Suppose now that $d_R<2888$. Then there are at most $2887$ $B_3$-branches adjacent to $R$. Removing them, leaves at least $n-63-7\cdot 2887 = n-20272$ vertices.
By Lemma \ref{degree k}, all degree-5 vertices in $T$ are contained in at most four $B_4$-branches and all degree-3 vertices are contained in a single subtree $T_x$ of some son $x$ of $R$. By (P1)--(P3), $x$ has smallest degree among all sons of $R$ (when it exists).

 Thus $T_{x_1},\dots, T_{x_{r-1}}$ are also subtrees of $T$. In particular, each such $k_i$ is at most 143. Thus, $d_R-1 \ge (n-20272-(1+7k_r))/(1+7\cdot 143)) \ge 873$. By Lemmas \ref{lemma:C} and \ref{lem:7k+8}, each $k_i$ ($2\le i<r$) is either 51, 52, or 53, and at least one of them is 51 or 52. 

We conclude that there is a $C_k$-branch in $T$ with $k\in \{51,52\}$.
Suppose that a $B_3$-branch is attached to the root. Now consider the change suggested in Figure \ref{fig:s_is_zero3}.

\begin{figure}[htb]
    \centering
    \includegraphics[width=0.65\textwidth]{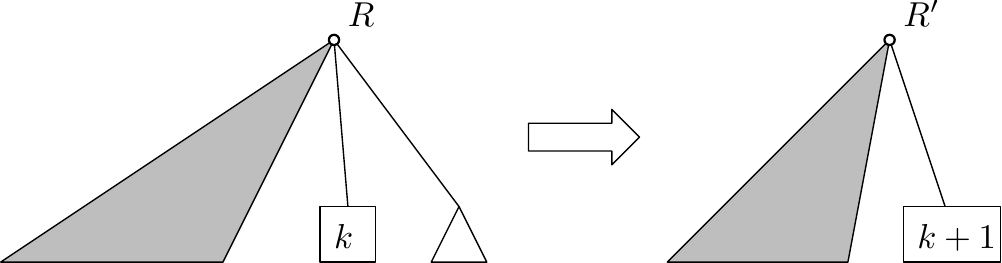}
    \caption{Merging $C_{k}$ and $B_3$ attached to the root.}
    \label{fig:s_is_zero3}
\end{figure}

If a neighbor of $R$ in the shaded part of Figure \ref{fig:s_is_zero3} has degree $a$, Lemmas \ref{lemma:C} and \ref{lem:k_is_51} imply that $4\leq a\leq 54$. Then the change of the ABC-contribution of the corresponding edge is $f(d_R,a)-f(d_R-1,a)$, which is decreasing in $a$ by Proposition \ref{prop:2}. Therefore the worst case happens when $k=52$ and $a=54$. Also note that by Lemma \ref{lem:k_is_51} there can be at most 364 copies of $C_{53}$ in $T$. So we have:
\begin{eqnarray*}
  \Delta(T,T') &\geq&  364f(d_R,54) + (d_R-365)f(d_R,53)+f(d_R,4)+52f(53,4)\\
&&  -  365f(d_R-1,54) - (d_R-366)f(d_R-1,53)-53f(54,4).
\end{eqnarray*}
Since the right hand side only depends on $d_R$ it is easy to check that for $d_R\geq 2948$, we have $ABC(T)-ABC(T')>0$.

To obtain a contradiction, it remains to prove that $d_R\geq 2948$. First of all, we may assume that $k_1\le 53$. By repeating the calculation from above: we first remove up to 63 vertices, next we remove $T_{x_r}$ which has at most $1+7\cdot53 = 365$ vertices.
For possible $C_{53}$-branches we remove one of their $B_3$-subtrees. Each of the resulting branches then contains at most 365 vertices. Thus $$d_R \ge 1 + (n - 63 - 365-7\cdot364)/365 >\frac{n}{365}-8\geq 2948.$$

The last inequality combined with Lemma \ref{lem:7k+8} also shows that most of the $C$-branches are $C_{52}$ and that at most 364 of these are $C_{51}$ or $C_{53}$.

We proved that there are no $B_3$-branches attached to the root. This means  that there are at least $ d_R -1 -4 -364 \geq \frac{n}{365} - 377$ $C_{52}$-branches.
\end{proof}

\begin{theorem}
\label{thm:large_structure}
For every ABC-extremal tree of order $n$, one can delete $O(1)$ vertices to obtain a subtree $T_r$ shown in Figure \ref{fig:main_struct large}, where $r=\lfloor n/365\rfloor - O(1)$. More precisely after deleting at most $63+o(1)$ vertices we are left with a tree whose root has degree $d$ satisfying $\tfrac{1}{365}n - (13+o(1)) \le d \le \tfrac{1}{365}n + (1+o(1))$, and all its sons are $C_{52}$-branches together with at most $364$ $C_k$-branches where $k$ is either $51$ or $53$ and possibly one additional $C_l$-branch where $l\leq 52$.
\end{theorem}

\begin{proof}
The $o(1)$ notation takes care of small values of $n$, so we may assume that $n\ge 1078940$, when the value $o(1)$ will be 0.
Then Lemma \ref{lem:s_is_zero} applies (in which case we remove at most 36 vertices in $B_4$ branches and at most 365 vertices in a subtree of the root of smallest degree in which some ending branches would be different from $B_3$).
\end{proof}

\begin{figure}[htb]
    \centering
    \includegraphics[width=0.375\textwidth]{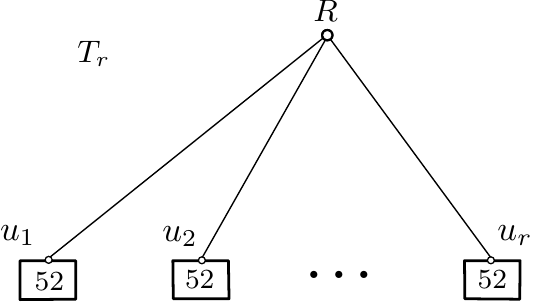}
    \caption{The structure of large ABC-minimal trees after deleting $O(1)$ vertices.}
    \label{fig:main_struct large}
\end{figure}

The estimates on the number of vertices of the tree used in the above theorems are very liberal. The transition to the desired form with mostly $C_{52}$-branches occurs much earlier. Let $n$ be the order of an ABC-minimal tree. As introduced in Figure \ref{fig:main_struct}, let $r$ be the number of $C_k$-branches adjacent to the root, and let $s$ be the number of $B_3$-branches adjacent to the root. For small values of $n$ the main structure is when $r=0$ which gives us the structure that was conjectured by Gutman \cite{ref24}. Simulation show that for larger $n$ we have positive values of $r$. To be more precise, we have positive $r$ when the following holds:
%
$n\equiv 0 \mod 7$ and $n\geq 525$,
$n\equiv 1 \mod 7$ and $n\geq 939$,
$n\equiv 2 \mod 7$ and $n\geq 422$,
$n\equiv 3 \mod 7$ and $n\geq 864$,
$n\equiv 4 \mod 7$ and $n\geq 508$,
$n\equiv 5 \mod 7$ and $n\geq 740$, or
$n\equiv 6 \mod 7$ and $n\geq 664$.

Also, for a very large $n$, Lemma \ref{lem:s_is_zero} shows that $s=0$, which is the structure conjectured in \cite{ourpaper}.

\begin{corollary}
\label{cor:gamma_n asympt}
Let $\gamma_n = \min\{ \abc(T) : |V(T)|=n\}$ and let $c_0 = \tfrac{1}{365}\sqrt{\tfrac{1}{53}} \bigl( 1 + 26\sqrt{55} + 156\sqrt{106}\, \bigr)$. Then
$$
   c_0 n - 365 c_0 + \tfrac{51}{2} \sqrt{\tfrac{1}{53}} - O(n^{-1}) \le \gamma_n 
   \le c_0n + 365 c_0 + \tfrac{51}{2} \sqrt{\tfrac{1}{53}} + O(n^{-1})
$$
and hence
$$
    \lim_{n\to \infty}\tfrac{1}{n}\gamma_n = c_0 \approx 0.67737178.
$$
\end{corollary}

\begin{proof}
Let us first establish the upper bound. Let $r=\lfloor(n-1)/365\rfloor$, and $t=n-365r-1$. Consider the tree $T_{r+1}$ with $r+1$ $C_{52}$-branches, and let $T$ be a tree of order $n$ obtained from $T_{r+1}$ by removing $365-t$ vertices from the last $C_{52}$-branch. The removal of the vertices can be done in such a way that
\begin{equation}
  \abc(T_r)\le \abc(T)\le \abc(T_{r+1}).
\label{eq:between Tr}
\end{equation}
To see this, note first that deleting any subset of vertices of degree 1 from the $C_{52}$-branch decreases the ABC-index. The same holds when removing any subset of $B_3$-branches and some degree-1 vertices. Finally, by removing the last vertex -- the root of the $C_{52}$-branch -- a short calculation shows that the ABC-index drops as long as $r$ is large enough (which we may assume).

Thus, it suffices to estimate $\abc(T_{r+1})$ (and $\abc(T_r)$), which we do next. Clearly,
\begin{eqnarray*}
  \tfrac{1}{r+1}\abc(T_{r+1}) &=& f(53,r+1) + 52 f(53,4) +156 f(4,2) + 156 f(2,1) \\
    &=& \sqrt{\tfrac{1}{53}} \sqrt{1+\tfrac{51}{r+1}} + 26\sqrt{\tfrac{55}{53}} + 156\sqrt{2}\\
    &=& \sqrt{\tfrac{1}{53}} \Bigl( 1 + \tfrac{51}{2(r+1)} + O(r^{-2}) \Bigr) + 26\sqrt{\tfrac{55}{53}} + 156\sqrt{2}.
\end{eqnarray*}
Since $r\le n/365$, we obtain
\begin{eqnarray*}
  \abc(T) &\le& \abc(T_{r+1}) \\
    &=&  \Bigl( \sqrt{\tfrac{1}{53}} + 26\sqrt{\tfrac{55}{53}} + 156\sqrt{2} \Bigr) r +
      \Bigl( \sqrt{\tfrac{1}{53}}\tfrac{51}{2} + \sqrt{\tfrac{1}{53}} + 26\sqrt{\tfrac{55}{53}} + 156\sqrt{2} + O(r^{-1}) \Bigr) \\
    &\le& c_0 n + 365 c_0 + \tfrac{51}{2} \sqrt{\tfrac{1}{53}} + O(n^{-1}).
\end{eqnarray*}

The lower bound follows from Theorem \ref{thm:large_structure} and using the same calculation as above (for $r$ instead $r+1$), and using the fact that $r\ge n/365 - 1$.
\end{proof}

\subsection*{Acknowledgement.}
The authors are grateful to an anonymous referee for carefully checking the results and for providing additional information and insight that simplified some of our proofs.

\end{document}